\documentclass[11pt]{article}

\usepackage{amsmath,amsfonts,amssymb,amscd,mathrsfs,amsthm} 
\usepackage{array}
\usepackage{paralist}
\usepackage{xcolor}
\usepackage{multicol}
\usepackage{multirow}
\usepackage{graphicx}
\usepackage{booktabs}
\usepackage{subfigure}
\usepackage{grffile}
\usepackage{float}   
\usepackage{mathrsfs}
\usepackage{cases}
\usepackage{lscape}
\usepackage{ulem}
\usepackage{soul,xcolor}

\newtheorem{proposition}{Proposition}
\newtheorem{assumption}{Assumption}
\newtheorem{lemma}{Lemma}

\newtheorem{remark}{Remark}
\newtheorem{theorem}{Theorem}[section]

\DeclareMathOperator{\Tr}{Tr}

\usepackage{mathtools}
\usepackage[font=small,labelfont=bf,labelsep=quad]{caption}
\usepackage[colorlinks,citecolor=blue,urlcolor=blue,linkcolor=blue]{hyperref}
\usepackage{setspace}
\usepackage[margin=1in]{geometry}

\usepackage{tikz}
\usetikzlibrary{arrows.meta, positioning, matrix}

\renewenvironment{thebibliography}[1]{\begin{oldthebibliography}{#1}\setlength{\baselineskip}{.8em}
\linespread{1.5}
\small
\setlength{\parskip}{0ex}\setlength{\itemsep}{.05em}}{\end{oldthebibliography}}

\mathtoolsset{
  showonlyrefs,
  mathic }
\title{Regularity of a Multidimensional Principal–Agent Problem with Separable Effort Costs}

\author{Shuaijie Qian \footnote{Department of Mathematics, The Hong Kong University of Science and Technology. sjqian@ust.hk.}\qquad  Guan Qiao \footnote{Department of Mathematics, The Hong Kong University of Science and Technology. gqiaoaa@connect.ust.hk.}
}

\begin{document}

\maketitle

\begin{abstract}
This paper studies the regularity of the value function arising from a multidimensional continuous-time principal-agent model with separable, nonquadratic effort costs. The associated stochastic control problem has the output and the agent’s continuation utility as state variables, and its Hamilton–Jacobi–Bellman equation is fully nonlinear and degenerate, with potentially unbounded coefficients. We address these difficulties by adding an independent regularization noise and bounding the effort. For the resulting problem, we establish classical regularity of the value function and show that the optimal effort is unique, positive and remains in a fixed compact subset, uniformly with respect to both the control restriction and the regularization parameter. These estimates allow us first to remove the control restriction and then to let the additional noise vanish. Consequently, we prove that the regularized value function converges to the original value function and conclude that the latter belongs locally to the Sobolev space \(W^{2,1}_{\infty, loc}\), thereby extending the regularity analysis to separable nonquadratic effort costs, for which the arguments yielding classical solutions in the quadratic-cost setting no longer apply.
\end{abstract}

\section{Introduction}
In a principal-agent effort model, a principal delegates a project to an agent whose unobservable effort affects the project's output. Since the effort cannot be directly observed, the principal proposes at the initial time a contract contingent on the project’s output process. The agent exerts effort by trading off the contract payoff against effort costs, while the principal designs the contract to balance the project output against the compensation paid to the agent. 

This Stackelberg game problem involves the optimization of the contract in a functional space. The continuous-time model is formulated by Holmstr{\"o}m and Milgrom \cite{holmstrom1987aggregation}. In a complementary direction, Spear and Srivastava \cite{spear1987repeated} introduce a recursive method using the agent's continuation/promised utility as a state variable. This approach is later extended to continuous-time models by DeMarzo and Sannikov \cite{demarzo2006optimal} and Sannikov \cite{sannikov2008continuous}. Following this direction, Cvitani{\'c} et al. \cite{cvitanic2009optimal} and Cvitani{\'c} and Zhang \cite{cvitanic2012contract} provide sufficient conditions for the optimal contract in a general continuous-time problem with quadratic effort costs. Their characterization relies on the existence of optimizers for an associated static optimization problem (see Proposition 3.1 of \cite{cvitanic2009optimal} or Section 6.2.3 of \cite{cvitanic2012contract}).

Subsequently, Cvitani{\'c} et al. \cite{cvitanic2018dynamic} developed a systematic dynamic-programming approach to general continuous-time principal-agent problems. Their framework reformulates the principal's problem as a standard stochastic control problem whose state variables include both the output process and the continuation utility. Under sufficient smoothness of the associated value function, their approach yields an optimal contract (see Theorem 3.9 and Proposition 5.4 in \cite{cvitanic2018dynamic}). 

Establishing the required regularity, however, remains difficult in general. To the best of our knowledge, results obtained through a regularity analysis of the value function are confined to special settings, including explicitly solvable examples in Cvitani{\'c} et al. \cite{cvitanic2018dynamic} and the one-dimensional model studied by Possama{\"\i} and Touzi \cite{possamai2025there}. An alternative approach is developed by Kr{\v{s}}ek and Possama{\"\i} \cite{krvsek2026randomisation}, who allow the agent to use measure-valued controls and employ compactification techniques to prove the existence of optimal contracts, thereby circumventing the typical challenge of degenerate partial differential equations (PDEs).
Although the associated value function of the principal-agent problem is characterized by an HJB equation, this PDE is degenerate and fully nonlinear due to the special structure of such a problem, making the existence theory for classical solutions inapplicable. In some special cases, only $W^{2, 1}_p$ regularity of solutions to such PDEs is guaranteed (see Krylov \cite{krylov1987nonlinear}, Chapter 7).

In recent work, Chen et al. \cite{chen2025optimal} establish the regularity of the value function for quadratic effort costs, thereby providing a PDE-based verification of the candidate optimal contract. However, their technique relies heavily on the special structure of the quadratic cost case and therefore does not extend directly to the multidimensional setting with nonquadratic cost.

In this paper, we establish the Sobolev regularity of the value function for a multidimensional principal-agent problem with separable nonquadratic effort costs. To address the degeneracy and potentially unbounded coefficients of the original HJB equation, we introduce an independent regularization noise term and restrict the admissible control set. The resulting HJB equation is uniformly parabolic and has bounded coefficients, allowing classical regularity theory to be applied. Crucially, we prove that the classical solution belongs locally to the Sobolev space $W^{2,1}_{\infty, loc}$ and the optimal control for this restricted problem is uniformly bounded. Additionally, the local estimates and bounds are all independent of both the control restriction and the regularization parameter. Finally, we establish the convergence of these classical solutions to the original value function. Together with the reflexivity of the Sobolev space \(W^{2,1}_p\) for finite \(p\) and Mazur's lemma, we thereby obtain the regularity of the latter.

The structure of the paper is as follows. In Section \ref{sec_model_setup}, we introduce the principal-agent problem and the associated stochastic control problem. Section \ref{sec_restrict} regularizes the state dynamics and restricts the admissible effort set; it then establishes the existence of a classical solution to the regularized HJB equation and the uniform boundedness of the optimal effort. Section \ref{sec_cvg_w^b} removes the effort restriction while preserving these regularity estimates. Section \ref{sec_cvg_w} lets the regularization parameter tend to zero and proves the local \(W^{2,1}_{\infty, loc}\) regularity of the original value function. Section \ref{sec_conclusion} concludes the paper.

\section{Model setup} \label{sec_model_setup}
In this section, we introduce the setup of the model. In Subsection \ref{subsection_PA}, we introduce the principal-agent problem following the model of Cvitani{\'c} et al. \cite{cvitanic2009optimal}. In Subsection \ref{subsection_control}, we introduce the associated control problem of interest. 

\subsection{Principal-agent problem} \label{subsection_PA}
Following Cvitani{\'c} et al. \cite{cvitanic2009optimal}, we assume that in the market there is a principal and an agent. 
The principal has a project which is delegated to the agent over the time interval $[0, T]$. The agent can control the output process by exerting effort. Precisely, we assume the project's output process is  $\left\{ \hat{x}(s) \right\}_{0\leq s \leq T}$ which obeys the following stochastic differential equation (SDE):
\begin{equation}
\begin{cases}
    d\hat{x}(s) = {\lambda}(s) ds + \sigma d \mathcal{B}(s), & \text{for } s \in (0, T], \\
    \hat{x}(0) = x,
\end{cases}
\label{dyn_x}
\end{equation}
Here, \(x\in\mathbb{R}^d\) and \(\sigma=\operatorname{diag}(\sigma_1,\ldots,\sigma_d)\in\mathbb{R}^{d\times d}\), where \(\sigma_i\neq0\) for every \(i\in\{1,\ldots,d\}\). Let \(\mathbb F^{\hat{x}}\) denote the augmented filtration generated by \(\hat{x}\). The agent chooses an effort process \({\Lambda} = \{ {\lambda}(s) \}_{0 \leq s \leq T}\) that is \(\mathbb F^{\hat{x}}\)-optional and takes values in $[0,\infty)^d$. Additionally, $\{\mathcal{B}(s)\}_{s \geq 0}$ is a standard $d$-dimensional Brownian motion on a filtered probability space $\big( \mathbb{S}, \mathscr{F}^\mathcal{B},\left\{\mathscr{F}^\mathcal{B}_s\right\}_{s \geq 0}, P\big)$ with $\mathcal{B}(0) = 0$ almost surely. The filtration $\left\{\mathscr{F}^\mathcal{B}_s\right\}_{s \geq 0}$ is
generated by the Brownian motion, is right-continuous and each $\mathscr{F}^\mathcal{B}_s$ contains all $P$-null sets of $\mathscr{F}^\mathcal{B}$. 

The principal can observe the output process rather than the agent's effort. In order to incentivize the agent to exert effort, the principal will determine at time $0$ a contract based on the output process to compensate for the agent's effort costs. Precisely, at time $T$, the principal will compensate the agent by an amount $\xi$, where $\xi$ is a random variable that is contingent on $\{\hat{x}(s)\}_{0 \leq s \leq T}$. Moreover, we say $\xi \in \mathcal{C}_0$ if
\begin{align}\label{regularity xi}
\mathbb{E}[|U_A(\xi)|^p] < \infty \qquad \text{for some} \ p >1,
\end{align}
where $U_A(x) {\in C^{2+\alpha^0}}: \mathbb{R} \rightarrow \mathbb{R}$, for some $\alpha^0 \in (0,1)$, is the agent's utility function satisfying
\begin{assumption} \label{assumptionUA}
    $$
    \frac{1}{c} < (U_A^{-1})'(x) < c, \quad 0 {\leq} (U_A^{-1})''(x) \leq c.
    $$
\end{assumption}
The agent's objective function is
$$
J_A^{\xi;{\Lambda}}(x, t) := \mathbb{E}\left[U_A\left(\xi\right)-\int_t^T c({\lambda}(s)) d s\bigg| \hat{x}(t) = x\right],
$$
where $c(\lambda): [0,\infty)^d \rightarrow \mathbb{R}$ is the cost function satisfying
\begin{assumption} \label{assumption_cost}
  (1) The function $c$ is additively separable, i.e.,
    \[
    c(\lambda) = \sum_{i=1}^{d} c_i(\lambda_i),
    \]
    for all $\lambda = (\lambda_1, \dots, \lambda_d)^\top \geq 0$.\\
    (2) For any integer $1 \leq i \leq d$, $c_i(0) = c'_i(0) = c''_i(0) = 0$, $c'_i(\lambda) > 0$ and $c''_i(\lambda) > 0$ for $\lambda > 0$, and $c^{(3)}_i(\lambda) \geq 0$. Additionally, each $c_i$ satisfies \\
    (a) 
    \quad $\log(c'_i(\lambda))$ is concave on $(0,\infty)$
    and \\
    (b) \quad $\lim\limits_{\lambda \rightarrow \infty} \frac{\log\left(c'_i(\lambda)\right)}{\lambda} = 0.$
\end{assumption}
\begin{remark}
    For each $i \in \{1,\dots,d\}$ and any $ \lambda^{(1)}, \lambda^{(2)} \geq 0$, we have
    \begin{equation} \label{general_mean_ineq}
    (c'_i)^{-1}\bigg(\frac{c'_i(\lambda^{(1)}) + c'_i(\lambda^{(2)})}{2}\bigg) \leq c^{-1}_i\bigg(\frac{c_i(\lambda^{(1)}) + c_i(\lambda^{(2)})}{2}\bigg).
    \end{equation}
    To this end, define $f(x) := c_i \circ (c'_i)^{-1} (x)$ and $x_1 := c'_i(\lambda^{(1)}), x_2 := c'_i(\lambda^{(2)})$. Then
    \begin{align}
        \eqref{general_mean_ineq} \Leftrightarrow \frac{x_1+x_2}{2} \leq c'_i \circ c^{-1}_i\bigg(\frac{c_i \circ (c'_i)^{-1}(x_1) + c_i \circ (c'_i)^{-1}(x_2)}{2}\bigg) = f^{-1}\left(\frac{f(x_1)+f(x_2)}{2}\right),
    \end{align}
    thus it is sufficient to show that $f(x)$ is convex. Indeed, defining $y := (c'_i)^{-1} (x)$, we obtain
    \begin{align}
        f'(x) = c'_i \circ (c'_i)^{-1} (x) [(c'_i)^{-1}]'(x) = \frac{c'_i \circ (c'_i)^{-1} (x) }{c''_i \circ (c'_i)^{-1}(x)} = \frac{c'_i(y) }{c''_i(y)} =: g(y),
    \end{align}
    and
    \begin{align}
        f''(x) = g'(y) \frac{dy}{dx} = g'(y) \frac{1}{c_i''\left((c'_i)^{-1}(x)\right)}.
    \end{align}
    Since $c''_i > 0$, it is sufficient to show that $g(y)$ is increasing. Indeed,
    \begin{align}
        g'(y) = \frac{(c''_i(y))^2 - c'_i(y) c^{(3)}_i(y)}{(c''_i(y))^2} \geq 0,
    \end{align}
    where the inequality follows from (a) in Assumption \ref{assumption_cost}.
\end{remark}
\begin{remark}
    Cost functions in this form satisfy Assumption \ref{assumption_cost}:
    $$
    c(\lambda) = \sum_{i=1}^{d} \frac{1}{p}\lambda_i^p, \quad p > 2.
    $$
    Notice that in this case, the inequality \eqref{general_mean_ineq} is the generalized mean inequality.
\end{remark}
Moreover, we require $\Lambda$ be chosen such that the SDE \eqref{dyn_x} permits a strong solution in $[t, T]$. Denote this set of admissible controls by $\mathcal{A}^0_\lambda(x, t)$. We additionally require that
\begin{equation} \label{assumption_lambda}
    \mathbb{E}\left[\int_t^T \left\|\lambda(s)\right\|^2 ds\right] < \infty, \,\mathbb{E}\left[\int_t^T \bigg(c\big(\lambda(s)\big)\bigg)^2 ds\right] < \infty, \,\mathbb{E}\left[\int_t^T \left\|\nabla c\big(\lambda(s)\big)\right\|^2 ds\right] < \infty.
\end{equation}
The agent's problem is to find an optimal effort process $\left\{{\lambda}(s)\right\}_{t\leq s \leq T}$ to maximize the objective function, i.e.,
$$
v^\xi(x, t) := \sup _{\Lambda\in \mathcal{A}^0_\lambda(x, t)} J_A^{\xi;{\Lambda}}(x, t), \qquad \xi \in \mathcal{C}_0.
$$
A control $\Lambda^*$ is called an $optimal \ response \ to \ contract \ \xi$ if $v^\xi(x, t) = J_A^{\xi;{\Lambda^*}}(x, t)$. We denote by $\mathcal{A}^*_\lambda(x, t, \xi)$ the collection of all such optimal controls.

We now turn to the principal's perspective. The agent has an outside opportunity, which is modeled as a reservation utility $y \in \mathbb{R}$. Thus, to persuade the agent into the game, the principal has to promise that the value of the agent's goal is no smaller than $y$. That is,
\begin{align}\label{contract0}
y(t) = v^\xi(x, t) \geq y.  
\end{align}
By the monotonicity of the utility functions, we only need to consider $ y(t) = y$. Thus, the principal can only choose a contract from the set
\begin{align}
    \Xi := \{\xi \in \mathcal{C}_0: \mathcal{A}^*_\lambda(x, t, \xi) \neq \emptyset \ \text{and} \ v^\xi(x, t) = y\}.
\end{align}
Following the standard convention, we assume that the agent chooses the response that is best for the principal when indifferent among optimal responses. Thus, the principal's problem is
\begin{align}
    u(x, y, t)=\sup _{\xi \in \Xi}\sup _{\Lambda \in \mathcal{A}^*_\lambda(x, t, \xi)}\mathbb{E}\bigg[U_P\bigg(\ell\left(\hat{x}(T)\right)-\xi\bigg) \bigg| {\hat{x}(t)} = {x}\bigg],
\end{align}
where $\ell\left(\cdot\right) {\in C^{2+\alpha^0}}: \mathbb{R}^d \rightarrow \mathbb{R}$ is a liquidation function satisfying
\begin{assumption} \label{assumptionliquid}
    $\nabla \ell(x) \geq \frac{1}{c_\ell} \mathbf{1}$, $- c_\ell \mathbf{I}  \preceq \nabla^2 \ell(x) \preceq 0$, and 
    $$
    |\ell(x)| \leq c_\ell(\|x\| + 1),
    $$
    for some constant $c_\ell > 0$,
\end{assumption}
and $U_P(x) {\in C^{2+\alpha^0}}: \mathbb{R} \rightarrow \mathbb{R}$ is the principal's utility function satisfying
\begin{assumption} \label{assumptionUP}
    $$
    U_P \leq 0, \quad U_P' > 0, \quad U_P'' \leq 0, \quad \left|\frac{U_P''}{U_P'}\right| \leq c_P
    $$
    for some constant $c_P > 1$.
\end{assumption}
\begin{remark} [Growth condition of $U_P'$ and {$U_P$}] \label{rmk_growth_U_P'}
    Fix $x_0 \in \mathbb{R}$. On the one hand, $\forall x \leq x_0$, we have
    \begin{align}
        &0 \geq \int_x^{x_0} \frac{U_P''}{U_P'}dt \geq \int_x^{x_0} - c_P dt\\
        \Rightarrow \ & 0 \geq \ln (\frac{U_P'(x_0)}{U_P'(x)}) \geq -c_P(x_0-x)\\
        \Rightarrow \ & U_P'(x_0) \leq U_P'(x) \leq U_P'(x_0) e^{c_P(x_0-x)}\\
        \Rightarrow \ & |U_P'(x)| \leq U_P'(x_0) e^{c_P|x-x_0|}.
    \end{align}
    On the other hand, $\forall x \geq x_0$, we have
    \begin{align}
        &0 \geq \int_{x_0}^x \frac{U_P''}{U_P'}dt \geq \int_{x_0}^x - c_P dt\\
        \Rightarrow \ & 0 \geq \ln (\frac{U_P'(x)}{U_P'(x_0)}) \geq -c_P(x-x_0)\\
        \Rightarrow \ & U_P'(x_0) \geq U_P'(x) \geq U_P'(x_0) e^{-c_P(x-x_0)}\\
        \Rightarrow \ & |U_P'(x)| \leq U_P'(x_0) e^{c_P|x-x_0|}.
    \end{align}
    Therefore, $|U_P'(x)| \leq U_P'(x_0) e^{c_P|x-x_0|}$.

    Since $\forall x \leq x_0$, we have $U_P'(x) \leq U_P'(x_0) e^{c_P(x_0-x)}$, integrating on both sides yields
    $$
    U_P(x) \geq U_P(x_0) - \frac{U_P'(x_0)}{c_P} e^{c_P(x_0 - x)} + \frac{U_P'(x_0)}{c_P}.
    $$
    Moreover, since $\forall x \geq x_0$, we have $U_P'(x) \geq U_P'(x_0) e^{-c_P(x-x_0)}$, it follows that
    $$
    U_P(x) \geq U_P(x_0) - \frac{U_P'(x_0)}{c_P} e^{c_P(x_0 - x)} + \frac{U_P'(x_0)}{c_P}.
    $$
\end{remark}

\subsection{The associated control problem} \label{subsection_control}
The above principal-agent problem motivates the study of a stochastic control formulation. Following Cvitani{\'c} et al. \cite{cvitanic2018dynamic,cvitanic2009optimal}, we take the following control problem as our object of study.

For $z\in[0,\infty)^d$, the agent's Hamiltonian associated with the output dynamics \eqref{dyn_x} is
\[
H(z):=\sup_{\lambda\in[0,\infty)^d}
\left\{\lambda^\top z-c(\lambda)\right\}.
\]
Under Assumption \ref{assumption_cost}, the Hamiltonian admits a unique maximizer characterized by
\[
z=\nabla c(\lambda^*).
\]
To be more precise, consider the process $\Lambda = \left\{{\lambda}(s)\right\}_{t\leq s \leq T}$ that is \(\mathbb F^{\hat{x}}\)-optional, takes values in $[0,\infty)^d$, satisfies \eqref{assumption_lambda},
and is such that the following stochastic differential equation admits a strong solution $\left\{ \left(\hat{x}^{\lambda}(s), y^{\lambda}(s)\right)\right\}_{t \leq s \leq T}$ on $[t,T]$.  
\begin{equation}
\begin{cases}
d \hat{x}^{\lambda}(s)= {\lambda}(s) d s+\sigma d \mathcal{B}(s), \\
d y^{\lambda}(s)=c\big(\lambda(s)\big) d s+ \left[\nabla c\left(\lambda(s)\right)\right]^\top \sigma d \mathcal{B}(s), \\
\hat{x}^{\lambda}(t)=x, \quad y^{\lambda}(t)=y.
\end{cases}
\label{auxiliary_noepsilon}
\end{equation}

We denote the set of all such admissible controls by $\mathcal{A}^0(x, y, t)$ and set
\begin{equation} \label{objective_noepsilon}
    w^0(x, y, t)=\sup _{\Lambda\in \mathcal{A}^0(x, y, t)} \mathbb{E}\left[U_P\big(\ell\left(\hat{x}^{\lambda}(T)\right)-U_A^{-1}(y^{\lambda}(T))\big)\right], \forall x\in \mathbb{R}^d, y\in \mathbb{R}, t \in [0, T].
\end{equation}
The following lemma establishes the necessary boundedness required to apply the dynamic programming principle.
\begin{lemma} \label{lemma_bdd_obj}
    For every $(x_0,y_0) \in \mathbb{R}^d \times \mathbb{R}$ and $R > 0$, $w^0$ is bounded on $\overline{Q}$, where $Q = B_R(x_0,y_0) \times (0,T)$ and $B_R(x_0,y_0) := \left\{(x,y) \in \mathbb{R}^d \times \mathbb{R} \mid \|x-x_0\|^2+(y-y_0)^2 < R^2\right\}$.
\end{lemma}
Together with the dynamic programming principle, Lemma \ref{lemma_bdd_obj} guarantees that $w^0(x, y, t)$ is the viscosity solution of the following PDE
\begin{equation} \label{pde_w_noepsilon}
    \begin{cases}
        w_t + \sup_{\lambda \in [0,\infty)^d}
        \Bigg\{
        \frac{1}{2} \Tr[\sigma \sigma^\top w_{xx}] + [\sigma\sigma^\top\nabla c(\lambda)]^\top w_{xy} + \frac{1}{2} \|\sigma^\top\nabla c(\lambda)\|^2 w_{yy} + \lambda^\top w_x + c(\lambda) w_y \Bigg\} = 0,\\
        w(x, y,T) = U_P(\ell\left(x\right) - U_A^{-1}(y)),\\
        \sup\limits_{t \in [0,T], x \in \mathbb{R}^d, y \in \mathbb{R}} \frac{|w(x,y,t)|}{e^{D(\|x\|+|y|)}} < \infty, \text{ for some constant } D > 0.
    \end{cases}
\end{equation}

We aim to establish that $w^0$ is a $W^{2, 1}_{\infty, loc}$ solution of the PDE \eqref{pde_w_noepsilon}. However, due to the degeneracy of the PDE operator and the unboundedness of the coefficients, the regularity of $w^0$ is not immediately guaranteed. 

To address this, we proceed in 3 steps. First, in Section \ref{sec_restrict}, we (1) introduce an additional noise term $\sqrt{\epsilon}\sigma d \mathcal{W}(s)$ to overcome degeneracy and (2) restrict the admissible control set to $[\frac{1}{c_\lambda},c_\lambda]^d$ to manage the unbounded coefficients. We then establish the regularity of the value function $w^b(x,y,t)$ and show that the optimal effort $\lambda^*_b$ is uniformly bounded, and that the uniform bound $\tilde{c}_\lambda$ is independent of the restriction $c_\lambda$ and $\epsilon$. Based on this, we prove that its partial derivatives are locally bounded up to second order, uniformly with respect to $c_\lambda$ and $\epsilon$. That is, $w^b \in W^{2,1}_\infty$ locally. The uniform boundedness of $\lambda^*_b$ and the regularity of $w^b$ represent the main contribution of our paper.
Second, in Section \ref{sec_cvg_w^b} we remove the restriction on the control set and show the convergence of the restricted value function $w^b(x,y,t)$ to the unrestricted one $w^\epsilon(x,y,t)$. Consequently, $w^\epsilon$ will inherit the regularity.
Finally, in Section \ref{sec_cvg_w}, we show that the unrestricted value function $w^\epsilon(x,y,t)$ in Section \ref{sec_cvg_w^b} converges to $w^0(x,y,t)$. Therefore, $w^0(x,y,t) \in W^{2,1}_p$ locally by the reflexivity of the Sobolev space $W^{2,1}_p$.

\section{Model restriction} \label{sec_restrict}
To overcome the degeneracy of the PDE and the unboundedness of the coefficients, we (1) modify the dynamic \eqref{auxiliary_noepsilon} by adding an independent noise and (2) restrict the admissible control set to be bounded. Precisely, we consider
\begin{equation}
\begin{cases}
d x^{\lambda}(s)= {\lambda}(s) d s+\sigma d \mathcal{B}(s) + \sqrt{\epsilon}\sigma d \mathcal{W}(s), \\
d y^{\lambda}(s)=c\big(\lambda(s)\big) d s+ \left[\nabla c\left(\lambda(s)\right)\right]^\top \sigma d \mathcal{B}(s), \\
x^{\lambda}(t)=x, \quad y^{\lambda}(t)=y, 
\end{cases}
\label{auxiliary}
\end{equation}
where $ 0 < \epsilon < 1$ is fixed, and $\mathcal{B}$ and $\mathcal{W}$ are independent standard $d$-dimensional Brownian motions.
Let \(\mathbb F^{\mathcal B,\mathcal W}\) denote the augmented filtration jointly generated by \(\mathcal B\) and \(\mathcal W\), and let \(\mathcal A(x,y,t)\) denote the set of \(\mathbb F^{\mathcal B,\mathcal W}\)-optional controls taking values in $[0,\infty)^d$, satisfying \eqref{assumption_lambda}, and are such that \eqref{auxiliary} admits a strong solution. Since \(\mathbb F^{\hat{x}} \subseteq \mathbb F^{\mathcal B,\mathcal W}\), we have \(\mathcal A^0(x,y,t)\subseteq\mathcal A(x,y,t)\).
Additionally, we define $\mathcal{A}^b(x, y, t) := \left\{ \Lambda \in \mathcal{A}(x, y, t) : \frac{1}{c_\lambda} \leq \lambda_i(s) \leq c_\lambda, \, \forall i = 1, \dots, d, \, \forall s \in [t, T] \right\}$ and set
\begin{equation} \label{objective_bdd_lamb}
    w^b(x, y, t) = \sup _{\Lambda\in \mathcal{A}^b(x, y, t)} \mathbb{E}\left[U_P\big(\ell\left(x^{\lambda}(T)\right)-U_A^{-1}(y^{\lambda}(T))\big)\right], \forall x\in \mathbb{R}^d, y\in \mathbb{R}, t \in [0, T].
\end{equation}

Notice that Lemma \ref{lemma_bdd_obj} applies analogously to establish the local boundedness of $w^b(x, y, t)$. Together with the dynamic programming principle, this ensures that $w^b(x, y, t)$ is the viscosity solution of the following PDE
\begin{equation} \label{pde_w_bdd_lamb}
    \begin{cases}
        w_t + \sup_{\lambda \in [\frac{1}{c_\lambda},c_\lambda]^d}
        \Bigg\{
        \frac{1+\epsilon}{2} \Tr[\sigma \sigma^\top w_{xx}] + [\sigma \sigma^\top \nabla c(\lambda)]^\top w_{xy}\\
        \qquad\qquad\qquad\qquad + \frac{1}{2} \|\sigma^\top \nabla c(\lambda)\|^2 w_{yy} + \lambda^\top w_x + c(\lambda) w_y \Bigg\} = 0,\\
        w(x, y,T) = U_P(\ell\left(x\right) - U_A^{-1}(y)),\\
        \sup\limits_{t \in [0,T], x \in \mathbb{R}^d, y \in \mathbb{R}} \frac{|w(x,y,t)|}{e^{D(\|x\|+|y|)}} < \infty, \text{ for some constant } D > 0.
    \end{cases}
\end{equation}
The corresponding first-order condition is given by
\begin{equation} \label{FOC_w^b}
    \nabla^2 c(\lambda_f) \sigma \sigma^\top  w^b_{xy} + \nabla^2 c(\lambda_f) \sigma \sigma^\top \nabla c(\lambda_f)w^b_{yy} + w^b_x + \nabla c(\lambda_f) w^b_y = 0,
\end{equation}

In Theorem \ref{thm_regularity_w^b}, we show that PDE \eqref{pde_w_bdd_lamb} admits a classical solution $w^b(x, y, t) \in C^{2+\alpha,\frac{2+\alpha}{2}}\bigg(\mathbb{R}^d \times \mathbb{R} \times [0,T]\bigg)$. Based on this, we also show that the F.O.C. \eqref{FOC_w^b} admits a unique and uniformly bounded solution. Thus, when $c_\lambda$ is large enough, this F.O.C. gives a unique and uniformly bounded solution $\lambda_f$ which is identical to the optimal effort $\lambda^*_b$. Additionally, we show that $w^b \in W^{2,1}_\infty$ locally, uniformly in $c_\lambda$ and $\epsilon$.

To this end, we have the following theorem for the regularity of $w^b(x,y,t)$.
\begin{theorem} \label{thm_regularity_w^b}
    PDE \eqref{pde_w_bdd_lamb} admits a classical solution $w^b(x, y, t) \in C_{loc}^{2+\alpha,\frac{2+\alpha}{2}}\bigg(\mathbb{R}^d \times \mathbb{R} \times [0,T]\bigg)$.
\end{theorem}

Building on the above theorem, we proceed to establish the uniform boundedness of the solution of the F.O.C. \eqref{FOC_w^b}. It follows from the additive separability of the cost function that for each $i \in \{1, \dots, d\}$, the $i$-th row of the F.O.C. \eqref{FOC_w^b} is:
\begin{equation} \label{FOC_w^b_explicit}
\sigma_i^2 c''_i(\lambda_{f,i}) w^b_{x_iy} + \sigma_i^2 c''_i(\lambda_{f,i}) c'_i(\lambda_{f,i}) w^b_{yy} + w^b_{x_i} + c'_i(\lambda_{f,i}) w^b_y = 0,
\end{equation}
where $\sigma_i \in \mathbb{R}$ denotes the $i$-th diagonal entry of $\sigma$ and $\lambda_{f,i}$ is the $i$-th component of $\lambda_f$. Consequently, it suffices to prove that each component of $\lambda_f$ lies in $[\frac{1}{\tilde{c}_\lambda},\tilde{c}_\lambda]$ for some constant $\tilde{c}_\lambda > 1$. 

The following proposition guarantees the crucial property of $\lambda_b^*$ and $w^b$, which constitutes the main contribution of our paper.
\begin{proposition} \label{prop_unique_lambda_w^b}  
    (1) The first-order condition \eqref{FOC_w^b} admits a unique solution $\lambda_f\in(0,\infty)^d$. Additionally, there exists a constant $\tilde{c}_\lambda > 1$, independent of $c_\lambda$ and $\epsilon$, such that $\lambda_f \in [\frac{1}{\tilde{c}_\lambda},\tilde{c}_\lambda]^d$. Consequently, the optimal effort $\lambda^*_b$ is unique and positive. Moreover, $\lambda^*_b \in [\frac{1}{\tilde{c}_\lambda},\tilde{c}_\lambda]^d,\quad \forall c_\lambda > 1, 0 < \epsilon < 1$.

    {(2) $w^b \in W^{2,1}_{\infty,loc}(\mathbb{R}^d \times \mathbb{R} \times [0,T])$. Moreover, $\forall (x_0,y_0) \in \mathbb{R}^d \times \mathbb{R}$ and $R > 0$, there exists a constant $C_Q > 0$, independent of $\epsilon$ and $c_\lambda$, such that $\|w^b\|_{W^{2,1}_\infty(\overline{Q})} \leq C_Q$, where $Q = B_R(x_0,y_0) \times (0,T)$ is the cylinder, and $B_R(x_0,y_0) := \{(x,y) \in \mathbb{R}^d \times \mathbb{R} \mid \|x-x_0\|^2+(y-y_0)^2 < R^2\}$.}
\end{proposition}
To prove (1), we determine the sign of the partial derivatives in the first-order condition \eqref{FOC_w^b} and show that they are all comparable. Precisely, for any integer $1 \leq j \leq d$, we show that (i) $\max \{w^b_{x_jx_j},w^b_{yy}\} \leq 0$, and {$(w^b_{x_jy})^2 \leq w^b_{x_jx_j}w^b_{yy}$}. (ii) $w^b_{x_j} > 0 > w^b_y$ and {$|\frac{w^b_y(x,y,t)}{w^b_{x_j}(x,y,t)}| \in [\frac{1}{c \cdot c_\ell},c \cdot c_\ell]$}. (iii) $w^b_{x_jx_j}$, $w^b_{yy}$ and $w^b_{x_jy}$ are bounded by $w^b_{x_j}$, {i.e., $w^b_{{x_j}{x_j}}(x,y,t) + \bar{c}_P w^b_{{x_j}}(x,y,t) \geq 0$, $w^b_{yy}(x,y,t) + c_H w^b_{{x_j}}(x,y,t) \geq 0$ and $\left|w^b_{{x_j}y}(x,y,t)\right| \leq c_H w^b_{x_j}(x,y,t)$ for some non-negative constants $\bar{c}_P$ and $c_H$}. (i) is guaranteed by Lemma \ref{lemma_sec_derivative_w^b} while (ii) and (iii) are guaranteed by Lemma \ref{lem_viscosity_partial_derivatives}. By (i)-(iii), together with the assumptions on the cost function, we see that, when $\lambda_{f,i}$ is sufficiently large, the negative terms in the first-order condition \eqref{FOC_w^b_explicit} will dominate, while for sufficiently small $\lambda_{f,i}$, the positive parts dominate. This leads to a contradiction when $\lambda_{f,i}$ is unbounded. The signs of the partial derivatives are also used to show the uniqueness. The proof of (1) will be divided into 4 steps. Additionally, (2) is guaranteed by (1).
\begin{proof}[Proof of Proposition \ref{prop_unique_lambda_w^b}]
    \textbf{Proof of (1):} 
    {Step 1:} We prove the concavity of $w^b$ to characterize the second-order partial derivatives of $w^b$, which guarantees that $w^b_{x_jy}$ can be dominated by $w^b_{x_j}$ later.
    
    This step is completed by the following lemma.
    \begin{lemma} \label{lemma_sec_derivative_w^b}
        $\forall 0 < \epsilon < 1, t \in [0,T]$, $w^b(x, y, t)$ is concave on $\mathbb{R}^d \times \mathbb{R}$. Equivalently, its Hessian matrix
        \[
        \nabla^2_{(x,y)} w^b = 
        \begin{pmatrix}
        w^b_{xx} & w^b_{xy} \\
        (w^b_{xy})^\top & w^b_{yy}
        \end{pmatrix}
        \]
        is negative semidefinite. In particular, $w^b_{yy} \leq 0$, $w^b_{x_j x_j} \leq 0$, $(w^b_{x_j y})^2 \leq w^b_{x_j x_j} w^b_{yy}$, and $(w_{x_jx_k}^b)^2 \leq w_{x_jx_j}^b w_{x_kx_k}^b$ for $j,k \in \{1,2,\dots,d\}$ and $j \neq k$.
    \end{lemma}
    
    {Step 2:} For any integer $1 \leq j \leq d$, we show that (a) $w^b_{x_j} > 0 > w^b_y$ and {$|\frac{w^b_y(x,y,t)}{w^b_{x_j}(x,y,t)}| \in [\frac{1}{c \cdot c_\ell},c \cdot c_\ell]$}, and (b) $w^b_{{x_j}{x_j}}$, $w^b_{yy}$, and $w^b_{{x_j}y}$ are bounded by $w^b_{x_j}$, {i.e., $w^b_{{x_j}{x_j}}(x,y,t) + \bar{c}_P w^b_{{x_j}}(x,y,t) \geq 0$, $w^b_{yy}(x,y,t) + c_H w^b_{{x_j}}(x,y,t) \geq 0$, and $\left|w^b_{{x_j}y}(x,y,t)\right| \leq c_H w^b_{x_j}(x,y,t)$ for some non-negative constants $\bar{c}_P$ and $c_H$}. 
    
    This step guarantees that for sufficiently large $\lambda_{f,i}$, the term $c'_i(\lambda_{f,i}) w_y$ dominates $w_{x_i}$ in the first-order condition \eqref{FOC_w^b_explicit}, while for sufficiently small $\lambda_{f,i}$, the opposite holds. The signs are used to yield a contradiction. Moreover, given (a), for sufficiently large $\lambda_{f,i}$, the second-order partial derivatives are dominated by $w_y$, whereas for sufficiently small $\lambda_{f,i}$, they are dominated by $w_{x_i}$. Therefore, the signs of $w_{x_i}$ and $w_y$ lead to the contradiction.

    Let $\Delta > 0$ and let $e_j \in \mathbb{R}^d$ denote the $j$-th standard basis vector. We define
    \begin{align} \label{finite_diff}
        p^{+}_{j,\Delta}(x,y,t) &= \frac{w^b(x+\Delta e_j,y,t)-w^b(x,y,t)}{\Delta},\\
        p^{-}_{j,\Delta}(x,y,t) &= \frac{w^b(x,y,t)-w^b(x-\Delta e_j,y,t)}{\Delta},\\
        q^{+}_\Delta(x,y,t) &= \frac{w^b(x,y+\Delta,t)-w^b(x,y,t)}{\Delta},\\
        q^{-}_\Delta(x,y,t) &= \frac{w^b(x,y,t)-w^b(x,y-\Delta,t)}{\Delta},\\
        s_{j,\Delta}(x,y,t) &= \frac{w^b(x+\Delta e_j,y,t)+w^b(x-\Delta e_j,y,t)-2w^b(x,y,t)}{\Delta^2},\\
        r_\Delta(x,y,t) &= \frac{w^b(x,y+\Delta,t)+w^b(x,y-\Delta,t)-2w^b(x,y,t)}{\Delta^2}.
    \end{align}
    It follows that 
    \begin{align*}
        w^b_{x_j}(x,y,t) &= \lim_{\Delta \rightarrow 0 }p^{-}_{j,\Delta}(x,y,t) = \lim_{\Delta \rightarrow 0 }p^{+}_{j,\Delta}(x,y,t),\\
        w^b_{y}(x,y,t) &= \lim_{\Delta \rightarrow 0 }q^{-}_\Delta(x,y,t) = \lim_{\Delta \rightarrow 0 }q^{+}_\Delta(x,y,t),\\
        w^b_{{x_j}{x_j}}(x,y,t) &= \lim_{\Delta \rightarrow 0 }s_{j,\Delta}(x,y,t),\\
        w^b_{yy}(x,y,t) &= \lim_{\Delta \rightarrow 0 }r_\Delta(x,y,t).
    \end{align*}

    We first show that $w^b_{x_j} > 0 > w^b_y$ for any integer $1 \leq j \leq d$. Since $\ell_{x_j}>0$, $(U_A^{-1})'>0$, and $U_P'>0$, evaluating the same optimal control yields
    \begin{align}
    w^b(x+\Delta e_j,y,t)&>w^b(x,y,t),\\
    w^b(x,y,t)&<w^b(x,y-\Delta,t).
    \end{align}
    Hence, by the concavity of $w^b$,
    \[
    w^b_{x_j}(x,y,t)\geq p^+_{j,\Delta}(x,y,t)>0>
    q^-_\Delta(x,y,t)\geq w^b_y(x,y,t).
    \]

    To proceed, we next establish the following lemma, which is crucial for the local uniform boundedness of $w^b$.
    {\begin{lemma} \label{lem_viscosity_partial_derivatives}
        For any integer $1 \leq j \leq d$, $p^{-}_{j,\Delta}(x,y,t)$, $q^{-}_\Delta(x,y,t)$ are sub-solutions, while $p^{+}_{j,\Delta}(x,y,t)$, $q^{+}_\Delta(x,y,t)$, $s_{j,\Delta}(x,y,t)$, $r_\Delta(x,y,t)$ are super-solutions of the PDE
        \begin{equation} \label{linear_pde_w^b}
        \begin{cases}
            H_t +
            \frac{1+\epsilon}{2} \Tr[\sigma \sigma^\top H_{xx}] + [\sigma \sigma^\top \nabla c(\lambda_b^*)]^\top H_{xy} + \frac{1}{2} \|\sigma^\top \nabla c(\lambda_b^*)\|^2 H_{yy} + (\lambda_b^*)^\top H_x + c(\lambda_b^*) H_y = 0,\\
            \sup\limits_{t \in [0,T], x \in \mathbb{R}^d, y \in \mathbb{R}} \frac{|H(x,y,t)|}{e^{D(\|x\|+|y|)}} < \infty, \text{ for some constant } D > 0,
            \end{cases}
        \end{equation}
        with terminal conditions
        \begin{align*}
            p^{+}_{j,\Delta}(x,y,T) &= \frac{U_P(\ell(x + \Delta e_j) - U_A^{-1}(y)) - U_P(\ell(x) - U_A^{-1}(y))}{\Delta}\\
            &= U_P'(\ell(x + \delta_1 e_j) - U_A^{-1}(y)) \ell_{x_j}(x + \delta_1 e_j) > 0,\\
            p^{-}_{j,\Delta}(x,y,T) &= \frac{U_P(\ell(x) - U_A^{-1}(y)) - U_P(\ell(x - \Delta e_j) - U_A^{-1}(y))}{\Delta} \\
            &= U_P'(\ell(x - \delta_2 e_j) - U_A^{-1}(y)) \ell_{x_j}(x - \delta_2 e_j),\\  
            q^{+}_\Delta(x,y,T) &= \frac{U_P(\ell(x) - U_A^{-1}(y + \Delta)) - U_P(\ell(x) - U_A^{-1}(y))}{\Delta} = -U_P'(\ell(x) - U_A^{-1}(\delta_3))(U_A^{-1})'(\delta_3),\\ 
            q^{-}_\Delta(x,y,T) &= \frac{U_P(\ell(x) - U_A^{-1}(y)) - U_P(\ell(x) - U_A^{-1}(y - \Delta))}{\Delta} = -U_P'(\ell(x) - U_A^{-1}(\delta_4))(U_A^{-1})'(\delta_4) < 0,\\ 
            s_{j,\Delta}(x,y,T) &= \frac{U_P(\ell(x + \Delta e_j) - U_A^{-1}(y)) + U_P(\ell(x - \Delta e_j) - U_A^{-1}(y)) - 2 U_P(\ell(x) - U_A^{-1}(y))}{\Delta^2}\\     &= U_P''(\ell(x + \delta_5 e_j) - U_A^{-1}(y)) \big(\ell_{x_j}(x + \delta_5 e_j)\big)^2 + U_P'(\ell(x + \delta_5 e_j) - U_A^{-1}(y)) \ell_{x_j x_j}(x + \delta_5 e_j),\\ 
            r_\Delta(x,y,T) &= \frac{U_P(\ell(x) - U_A^{-1}(y + \Delta)) + U_P(\ell(x) - U_A^{-1}(y - \Delta)) - 2 U_P(\ell(x) - U_A^{-1}(y))}{\Delta^2}\\     &= U_P''(\ell(x) - U_A^{-1}(\delta_6))\cdot \big((U_A^{-1})'(\delta_6)\big)^2 - U_P'(\ell(x) - U_A^{-1}(\delta_6))\cdot(U_A^{-1})''(\delta_6),
        \end{align*}
        for some $\delta_1, \delta_2 \in (0, \Delta)$, $\delta_3 \in (y,y+\Delta)$, $\delta_4 \in (y-\Delta,y)$, $\delta_5 \in (-\Delta,\Delta)$, $\delta_6 \in (y-\Delta,y+\Delta)$.
    \end{lemma}}
    \begin{remark} \label{rmk_CP}
        In what follows, we repeatedly use comparison principles for \eqref{pde_w_bdd_lamb}, including its counterpart with $\epsilon=0$, the linear PDE \eqref{linear_pde_w^b}, and the bounded-cylinder problem \eqref{pde_w_bdd_lamb_local}. The proof of the comparison principle in Appendix F of Chen et al. \cite{chen2025optimal} can be adapted to all these settings: the exponential barrier handles the unbounded-domain equations with bounded coefficients and solutions of at most exponential growth, while no barrier at spatial infinity is needed on bounded cylinders.
    \end{remark}

Before we compare $w^b_{x_j}(x,y,t)$ and $w^b_y(x,y,t)$, we define $\mu(\eta) := U_P(z+\eta)$, where $z = \ell(x)-U_A^{-1}(y)$ and $\eta \in \mathbb{R}$. This function is crucial to close Step 2. It follows from Assumption \ref{assumptionUP} that
\begin{align}
    &U_P''(z+\eta) + c_P U_P'(z+\eta) \geq 0\\
    \Rightarrow \  &e^{c_P\eta} U_P''(z+\eta) + c_P e^{c_P\eta} U_P'(z+\eta) \geq 0\\
    \Rightarrow \ &\left(e^{c_P\eta} U_P'(z+\eta)\right)' \geq 0\\
    \Rightarrow \ &\left(e^{c_P\eta} \mu'(\eta)\right)' \geq 0. \label{monotone_mu}
\end{align}
Hence, $e^{c_P\eta} \mu'(\eta)$ is non-decreasing on $\eta \in \mathbb{R}$.

By Lemma \ref{lem_viscosity_partial_derivatives}, we have
\begin{align}     
p^{+}_{j,\Delta}(x,y,T) &= U_P'(\ell(x + \delta_1 e_j)- U_A^{-1}(y)) \ell_{x_j}(x + \delta_1 e_j) \\     
&= U_P'(\ell(x)- U_A^{-1}(y)  + \ell(x + \delta_1 e_j) - \ell(x)) \ell_{x_j}(x + \delta_1 e_j) \\     
&= \mu'(\ell(x + \delta_1 e_j) - \ell(x)) \ell_{x_j}(x + \delta_1 e_j). \end{align}
Let $\Delta_\ell = \ell(x + \delta_1 e_j) - \ell(x)$. By the Mean Value Theorem, there exists some $\delta_1' \in (0, \delta_1)$ such that $\Delta_\ell = \ell_{x_j}(x + \delta_1' e_j) \delta_1$. Since $\delta_1 \in (0,\Delta)$, it follows from Assumption \ref{assumptionliquid} ($\ell_{x_j}(x) \leq c_\ell$) that $0 < \Delta_\ell < c_\ell \Delta$. Furthermore, Assumption \ref{assumptionliquid} guarantees $\ell_{x_j}(x) \geq 1/c_\ell$, meaning $1/\ell_{x_j}(x) \leq c_\ell$. It follows from the monotonicity of $e^{c_Pt} \mu'(t)$ that
\begin{align}
e^{c_P(-c\Delta)} \mu'(-c\Delta) &\leq e^{c_P\Delta_\ell} \mu'(\Delta_\ell) \\  
&< e^{c_P c_\ell \Delta} \mu'(\Delta_\ell) \\  
&= e^{c_P c_\ell \Delta} \frac{p^{+}_{j,\Delta}(x,y,T)}{\ell_{x_j}(x + \delta_1 e_j)} \\  
&\leq c_\ell e^{c_P c_\ell \Delta} p^{+}_{j,\Delta}(x,y,T),\\  \Rightarrow \ \mu'(-c\Delta) &\leq c_\ell e^{c_P(c + c_\ell)\Delta} p^{+}_{j,\Delta}(x,y,T). 
\end{align}
Additionally, since $\delta_3 \in (y,y+\Delta)$, we have $-\Delta < y-\delta_3 < 0$. It follows from the Mean Value Theorem and Assumption \ref{assumptionUA} that
\begin{align}
U_A^{-1}(y) - U_A^{-1}(\delta_3) &= (U_A^{-1})'(\delta_7) (y-\delta_3) > -c \Delta
\end{align}
for some $\delta_7 \in (y,\delta_3)$. Moreover, we have
\begin{align}
q^{+}_\Delta(x,y,T) &= -U_P'(\ell(x) - U_A^{-1}(\delta_3))(U_A^{-1})'(\delta_3)\\     
&> -c U_P'(\ell(x) - U_A^{-1}(y) + U_A^{-1}(y) -U_A^{-1}(\delta_3))\\    &\geq -c U_P'(\ell(x) - U_A^{-1}(y) -c\Delta)\\     
&= -c \mu'(-c\Delta)\\     
&\geq -c \cdot c_\ell e^{c_P(c + c_\ell)\Delta} p^{+}_{j,\Delta}(x,y,T). 
\end{align}
Therefore, we obtain
\begin{align} 
q^{+}_\Delta(x,y,T) + c \cdot c_\ell e^{c_P(c + c_\ell)\Delta} p^{+}_{j,\Delta}(x,y,T) > 0. 
\end{align}
It follows from the comparison principle for PDE \eqref{linear_pde_w^b} that
\begin{align}
q^{+}_\Delta(x,y,t) + c \cdot c_\ell e^{c_P(c + c_\ell)\Delta}  p^{+}_{j,\Delta}(x,y,t) \geq 0. 
\end{align}
Taking the limit as $\Delta \rightarrow 0$, we have
\begin{align}
w^b_y(x,y,t) + c \cdot c_\ell w^b_{x_j}(x,y,t) \geq 0 \Rightarrow \frac{w^b_y(x,y,t)}{w^b_{x_j}(x,y,t)} \geq -c \cdot c_\ell, \quad \forall (x,y,t) \in \mathbb{R}^d \times \mathbb{R} \times [0,T]. 
\end{align}
Similarly, by Lemma \ref{lem_viscosity_partial_derivatives}, we have
\begin{align}
p^{-}_{j,\Delta}(x,y,T) &= U_P'(\ell(x - \delta_2 e_j) - U_A^{-1}(y)) \ell_{x_j}(x - \delta_2 e_j)\\
&= U_P'(\ell(x) - U_A^{-1}(y) + \ell(x - \delta_2 e_j) - \ell(x)) \ell_{x_j}(x - \delta_2 e_j) \\
&= \mu'(\ell(x - \delta_2 e_j) - \ell(x)) \ell_{x_j}(x - \delta_2 e_j). \end{align}
Let $\Delta_\ell = \ell(x - \delta_2 e_j) - \ell(x)$. By the Mean Value Theorem, there exists some $\delta_2' \in (0, \delta_2)$ such that $\Delta_\ell = -\ell_{x_j}(x - \delta_2' e_j) \delta_2$. Since $\delta_2 \in (0,\Delta)$, it follows from Assumption \ref{assumptionliquid} ($\ell_{x_j}(x) \leq c_\ell$) that $-c_\ell \Delta < \Delta_\ell < 0$. Furthermore, Assumption \ref{assumptionliquid} guarantees $\ell_{x_j}(x) \leq c_\ell$, which implies $\mu'(\Delta_\ell) = \frac{p^{-}_{j,\Delta}(x,y,T)}{\ell_{x_j}(x - \delta_2 e_j)} \geq \frac{1}{c_\ell} p^{-}_{j,\Delta}(x,y,T)$. Therefore, it follows from the monotonicity of $e^{c_Pt} \mu'(t)$ that
\begin{align}
e^{c_P(c\Delta)} \mu'(c\Delta) &\geq e^{c_P\Delta_\ell} \mu'(\Delta_\ell)\\
&> e^{-c_P c_\ell \Delta} \mu'(\Delta_\ell)\\
&\geq \frac{1}{c_\ell} e^{-c_P c_\ell \Delta} p^{-}_{j,\Delta}(x,y,T)\\
\Rightarrow \ \mu'(c\Delta) &\geq \frac{1}{c_\ell} e^{-c_P(c + c_\ell)\Delta} p^{-}_{j,\Delta}(x,y,T). 
\end{align}
Additionally, since $\delta_4 \in (y-\Delta,y)$, we have $0 < y-\delta_4 < \Delta$. It follows from the Mean Value Theorem and Assumption \ref{assumptionUA} that
\begin{align}
U_A^{-1}(y) - U_A^{-1}(\delta_4) &= (U_A^{-1})'(\delta_8) (y-\delta_4) < c \Delta
\end{align}
for some $\delta_8 \in (\delta_4,y)$.
Moreover, we have
\begin{align}
q^{-}_\Delta(x,y,T) &= -U_P'(\ell(x) - U_A^{-1}(\delta_4))(U_A^{-1})'(\delta_4)\\
&< -\frac{1}{c} U_P'(\ell(x) - U_A^{-1}(y) + U_A^{-1}(y) -U_A^{-1}(\delta_4))\\
&\leq -\frac{1}{c} U_P'(\ell(x) - U_A^{-1}(y) + c\Delta)\\
&= -\frac{1}{c} \mu'(c\Delta)\\
&\leq -\frac{1}{c \cdot c_\ell} e^{-c_P(c + c_\ell)\Delta} p^{-}_{j,\Delta}(x,y,T). 
\end{align}
Therefore,
\begin{align}
q^{-}_\Delta(x,y,T) + \frac{1}{c \cdot c_\ell} e^{-c_P(c+c_\ell)\Delta} p^{-}_{j,\Delta}(x,y,T) < 0.
\end{align}
It follows from the comparison principle for PDE \eqref{linear_pde_w^b} that
\begin{align}
    q^{-}_\Delta(x,y,t) + \frac{1}{c \cdot c_\ell} e^{-c_P(c+c_\ell)\Delta} p^{-}_{j,\Delta}(x,y,t) \leq 0.
\end{align}
Taking the limit as $\Delta \rightarrow 0$, we have
\begin{align} \label{compare_w_x_w_y_upper}
    w^b_y(x,y,t) + \frac{1}{c \cdot c_\ell} w^b_{x_j}(x,y,t) \leq 0 \Rightarrow \frac{w^b_y(x,y,t)}{w^b_{x_j}(x,y,t)} \leq -\frac{1}{c \cdot c_\ell} \quad \forall (x,y,t) \in \mathbb{R}^d \times \mathbb{R} \times [0,T].
\end{align}
Thus, $\left|\frac{w^b_y(x,y,t)}{w^b_{x_j}(x,y,t)}\right| \in {[\frac{1}{c \cdot c_\ell},c \cdot c_\ell]}$.

To compare the second-order partial derivatives with $w^b_{x_j}(x,y,t)$,  we define $\tilde{\mu}(s) := U_P(\ell(x+s e_j) - U_A^{-1}(y))$ for $s \in \mathbb{R}$. 
To establish the monotonicity of $\tilde{\mu}'$, we compute its derivatives:
\begin{align}
\tilde{\mu}'(s) &= U_P'(z_s) \ell_{x_j}(x+s e_j), \\     \tilde{\mu}''(s) &= U_P''(z_s) \big(\ell_{x_j}(x+s e_j)\big)^2 + U_P'(z_s) \ell_{x_j x_j}(x+s e_j), 
\end{align}
where $z_s = \ell(x+s e_j) - U_A^{-1}(y)$. By Assumption \ref{assumptionliquid}, we have $1/c_\ell \leq \ell_{x_j} \leq c_\ell$ and $-c_\ell \leq \ell_{x_j x_j} \leq 0$. It follows from Assumption \ref{assumptionUP} that we can define a sufficiently large constant $\bar{c}_P > 0$ such that
\begin{align}
\tilde{\mu}''(s) + \bar{c}_P \tilde{\mu}'(s) &= U_P''(z_s) \ell_{x_j}^2 + U_P'(z_s) \ell_{x_j x_j} + \bar{c}_P U_P'(z_s) \ell_{x_j}\\
&\geq -c_P U_P'(z_s) c_\ell^2 - c_\ell U_P'(z_s) + \bar{c}_P U_P'(z_s) \left(\frac{1}{c_\ell}\right)\\
&= U_P'(z_s) \left[ \frac{\bar{c}_P}{c_\ell} - c_P c_\ell^2 - c_\ell \right] \geq 0. 
\end{align}
This implies that $\left(e^{\bar{c}_P s} \tilde{\mu}'(s)\right)' \geq 0$, meaning $e^{\bar{c}_P s} \tilde{\mu}'(s)$ is non-decreasing on $s \in \mathbb{R}$. It follows that
for any $0 \leq a \leq \Delta$, we have
\begin{align}
e^{\bar{c}_P a} \tilde{\mu}'(a) \geq e^{\bar{c}_P(a-\Delta)} \tilde{\mu}'(a-\Delta) \Rightarrow e^{\bar{c}_P\Delta} \tilde{\mu}'(a) \geq \tilde{\mu}'(a-\Delta). 
\end{align}

Notice that we have
\begin{align}
p^{+}_{j,\Delta}(x,y,T) &= \frac{\tilde{\mu}(\Delta)-\tilde{\mu}(0)}{\Delta} = \frac{1}{\Delta}\int_0^\Delta \tilde{\mu}'(a) da,\\
s_{j,\Delta}(x,y,T) &= \frac{\tilde{\mu}(\Delta)+\tilde{\mu}(-\Delta)-2\tilde{\mu}(0)}{\Delta^2}\\
&= \frac{\tilde{\mu}(\Delta)-\tilde{\mu}(0)}{\Delta^2} - \frac{\tilde{\mu}(0)-\tilde{\mu}(-\Delta)}{\Delta^2}\\
&= \frac{1}{\Delta^2} \left(\int_0^\Delta \tilde{\mu}'(a) da - \int_{-\Delta}^0 \tilde{\mu}'(b) db\right)\\
&=\frac{1}{\Delta^2} \int_0^\Delta \left(\tilde{\mu}'(a) - \tilde{\mu}'(a-\Delta)\right) da. 
\end{align}
Therefore, we have
\begin{align}
s_{j,\Delta}(x,y,T) &= \frac{1}{\Delta^2} \int_0^\Delta \left(\tilde{\mu}'(a) - \tilde{\mu}'(a-\Delta)\right) da\\
&\geq \frac{1}{\Delta^2} (1-e^{\bar{c}_P\Delta}) \int_0^\Delta \tilde{\mu}'(a) da\\
&= \frac{1-e^{\bar{c}_P\Delta}}{\Delta} \frac{\tilde{\mu}(\Delta)-\tilde{\mu}(0)}{\Delta}\\
&=\frac{1-e^{\bar{c}_P\Delta}}{\Delta} p^{+}_{j,\Delta}(x,y,T).
\end{align}
By the comparison principle for PDE \eqref{linear_pde_w^b}, we have
$$
s_{j,\Delta}(x,y,t) + \frac{e^{\bar{c}_P\Delta}-1}{\Delta} p^{+}_{j,\Delta}(x,y,t) \geq 0.
$$
Taking the limit as $\Delta \rightarrow 0$, we conclude by L'H\^opital's rule that
\begin{align} \label{compare_w_xx_w_x}
    w^b_{x_jx_j}(x,y,t) + \bar{c}_P w^b_{x_j}(x,y,t) \geq 0 \quad \forall (x,y,t) \in \mathbb{R}^d \times \mathbb{R} \times [0,T].
\end{align}

Let $c_H := c^2c_P+c > c_P$. It follows from Lemma \ref{lem_viscosity_partial_derivatives} that
\begin{align}
    r_\Delta(x,y,T) &= U_P''(\ell(x)- U_A^{-1}(\delta_6)) \cdot \big((U_A^{-1})'(\delta_6)\big)^2 - U_P'(\ell(x)- U_A^{-1}(\delta_6))\cdot(U_A^{-1})''(\delta_6)\\
    &\geq -c_H \cdot U_P'(\ell(x) - U_A^{-1}(\delta_6)).
\end{align}
Since $\delta_6 \in (y-\Delta,y+\Delta)$, we have $|y-\delta_6| < \Delta$.
It follows from the Mean Value Theorem and Assumption \ref{assumptionUA} that
\begin{align}
    U_A^{-1}(y) - U_A^{-1}(\delta_6) &= (U_A^{-1})'(\delta_9) (y-\delta_6) > -c \Delta
\end{align}
for some $\delta_9$ lying between $y$ and $\delta_6$. Therefore, we have
\begin{align}
    r_\Delta(x,y,T) &\geq -c_H \cdot U_P'(\ell(x)- U_A^{-1}(\delta_6)) \geq -c_H \cdot U_P'(\ell(x)- U_A^{-1}(y)-c\Delta) = -c_H \cdot \mu'(-c\Delta).
\end{align}
Moreover, for any $a \geq 0$, it follows from \eqref{monotone_mu} that
\begin{align}
 e^{c_Pa} \mu'(a) \geq e^{c_P(-c\Delta)} \mu'(-c\Delta) \Rightarrow \mu'(a) \geq e^{-c_P(a+c\Delta)} \mu'(-c\Delta).
\end{align}
Denote $\Delta_\ell = \ell(x+\Delta e_j) - \ell(x)$. It follows from the Mean Value Theorem and Assumption \ref{assumptionliquid} ($\ell_{x_j}(x) \leq c_\ell$) that $\frac{\Delta}{c_\ell} \leq \Delta_\ell \leq c_\ell \Delta$.
Therefore, we have
\begin{align}
    p^{+}_{j,\Delta}(x,y,T) = \frac{1}{\Delta}\int_0^{\Delta_\ell} \mu'(a) da &\geq \frac{1}{c_\ell \Delta_\ell}\int_0^{\Delta_\ell} e^{-c_P(a+c\Delta)} \mu'(-c\Delta) da\\
    &\geq e^{-c_P\Delta(c_\ell+c)} \frac{1}{c_\ell \Delta_\ell} \int_0^{\Delta_\ell} \mu'(-c\Delta) da\\
    &= e^{-c_P\Delta(c_\ell+c)} \frac{1}{c_\ell} \mu'(-c\Delta)\\
    &\geq e^{-c_P\Delta(c_\ell+c)} \cdot \frac{1}{c_\ell} \cdot \frac{-1}{c_H} \cdot r_\Delta(x,y,T).
\end{align}
It follows from the comparison principle for PDE \eqref{linear_pde_w^b} that
$$
r_\Delta(x,y,t) + c_\ell c_H e^{c_P\Delta(c_\ell+c)} p^{+}_{j,\Delta}(x,y,t) \geq 0.
$$
Taking the limit as $\Delta \rightarrow 0$, we conclude that
$$
w^b_{yy}(x,y,t) + c_\ell c_H w^b_{x_j}(x,y,t) \geq 0 \quad \forall (x,y,t) \in \mathbb{R}^d \times \mathbb{R} \times [0,T].
$$ 
Therefore,
\begin{align}
    \left|w^b_{x_jy}(x,y,t)\right| \leq \sqrt{w^b_{x_jx_j}(x,y,t)w^b_{yy}(x,y,t)} \leq c_H w^b_{x_j}(x,y,t),
\end{align}
where, without loss of generality, we redefine $c_H = \max \{\bar{c}_P, c_H c_\ell\}$.

{Step 3:} We show that the first-order condition has at most one
solution on $[0,\infty)^d$, and that any such solution is positive. This is guaranteed by the signs of the partial derivatives.
    
    For any integer $1 \leq i \leq d$, denote
    $\psi_i(\lambda_i) = \sigma_i^2 c''_i(\lambda_i) w^b_{x_iy} + \sigma_i^2 c''_i(\lambda_i) c'_i(\lambda_i) w^b_{yy} + w^b_{x_i} + c'_i(\lambda_i) w^b_y$. It follows immediately from Theorem \ref{thm_regularity_w^b} that $\psi_i(\lambda_i)$ is continuous and differentiable.
    Then
    $$
    \psi'_i(\lambda_i) = \sigma_i^2 c^{(3)}_i(\lambda_i) [w^b_{x_iy} + c'_i(\lambda_i)w^b_{yy}]+ \sigma_i^2 \left(c''_i(\lambda_i)\right)^2 w^b_{yy} + c''_i(\lambda_i) w^b_y.
    $$
    Let $\lambda_{f,i}$ be any zero of $\psi_i$ on $[0,\infty)$, i.e., $\psi_i(\lambda_{f,i}) = 0$. Since \(\psi_i(0)= w_{x_i}^b>0\), we have $\lambda_{f,i} > 0$ and $c''_i(\lambda_{f,i}) > 0$. It follows from F.O.C. \eqref{FOC_w^b_explicit} that
    \begin{align*}
        \psi'_i(\lambda_{f,i}) &= -\frac{c^{(3)}_i(\lambda_{f,i})}{c''_i(\lambda_{f,i})} [w^b_{x_i} + c'_i(\lambda_{f,i}) w^b_y] + \sigma_i^2 \big(c''_i(\lambda_{f,i})\big)^2 w^b_{yy} + c''_i(\lambda_{f,i}) w^b_y\\
        &= w^b_y \bigg\{-\frac{c^{(3)}_i(\lambda_{f,i}) c'_i(\lambda_{f,i})}{c''_i(\lambda_{f,i})} + c''_i(\lambda_{f,i})\bigg\} + \sigma_i^2 \big(c''_i(\lambda_{f,i})\big)^2 w^b_{yy} -\frac{c^{(3)}_i(\lambda_{f,i})}{c''_i(\lambda_{f,i})} w^b_{x_i}\\
        & < 0,
    \end{align*}
    where the last inequality follows from Assumption \ref{assumption_cost}, Step 2 and Lemma \ref{lemma_sec_derivative_w^b}.

    Now suppose there exist $\lambda^{(1)}_i$ and $\lambda^{(2)}_i$ such that $\lambda^{(1)}_i < \lambda^{(2)}_i$ and $\psi_i(\lambda^{(1)}_i) = \psi_i(\lambda^{(2)}_i) = 0$. Since $\psi_i'(\lambda_i)<0$ at every zero, all zeros of $\psi_i$ are isolated. Hence, we may choose $\lambda_i^{(1)}$ and $\lambda_i^{(2)}$ to be two consecutive zeros. By continuity, $\psi_i$ has a constant sign on $(\lambda_i^{(1)},\lambda_i^{(2)})$.
    Since $\psi'_i(\lambda^{(1)}_i) < 0$, for every $\epsilon_0 > 0$, there exists $\lambda^{(c)}_i \in (\lambda^{(1)}_i,\lambda^{(1)}_i+\epsilon_0)$ such that $\psi_i(\lambda^{(c)}_i) < 0$. Hence, $\forall \lambda_i \in (\lambda^{(1)}_i,\lambda^{(2)}_i)$, we have $\psi_i(\lambda_i) < 0$, which contradicts $\psi'_i(\lambda^{(2)}_i) < 0$.

    Therefore, the first-order condition has at most one
    solution on $[0,\infty)^d$, and any such solution is positive.

{Step 4:} We establish the existence and uniform boundedness of $\lambda_f$, i.e., $\lambda_f \in [\frac{1}{\tilde{c}_\lambda},\tilde{c}_\lambda]^d$ for some constant $\tilde{c}_\lambda > 1$, independent of $c_\lambda$ and $\epsilon$. Since \(\psi_i(0) = w_{x_i}^b>0\), the existence follows from the Intermediate Value Theorem by showing that $\psi_i(\lambda_i) < 0$ when $\lambda_i$ is large enough. To show the uniform boundedness, we show $\lambda_{f,i} \in [\frac{1}{\tilde{c}_\lambda},\tilde{c}_\lambda]$ for any integer $i$ such that $1 \leq i \leq d$.
Therefore, together with Steps 2 and 3, it follows that $\lambda^*_b$ is unique and positive. Moreover, $\lambda^*_b \in [\frac{1}{\tilde{c}_\lambda},\tilde{c}_\lambda]^d,\quad \forall c_\lambda > 1, 0 < \epsilon < 1$.

We first show the existence of $\lambda_f$. Let $\bar{c}_i$ be the smallest constant satisfying the following two conditions: 
$\frac{c_i'(\bar{c}_i)}{c_i''(\bar{c}_i)} \geq 2 \sigma_i^2 c c_H c_\ell$ and $c_i'(\bar{c}_i) \geq 2 c c_\ell$. Notice that $\bar{c}_i$ is independent of $c_\lambda$ and $\epsilon$ and the existence of $\bar{c}_i$ is guaranteed by Assumption \ref{assumption_cost}. Moreover, Assumption \ref{assumption_cost} guarantees that $\frac{c_i'(\bar{c}_i+1)}{c_i''(\bar{c}_i+1)} \geq 2 \sigma_i^2 c c_H c_\ell$ and $c_i'(\bar{c}_i+1) > 2 c c_\ell$. Then it follows from Steps 2 and 3 that
\begin{align}
\sigma_i^2 c''_i(\bar{c}_i+1) w^b_{x_iy} + \frac{c'_i(\bar{c}_i+1)}{2}w^b_y &\leq \sigma_i^2 c''_i(\bar{c}_i+1)c_H w^b_{x_i} + \frac{c'_i(\bar{c}_i+1)}{2}w^b_y\\
&\leq \sigma_i^2 c''_i(\bar{c}_i+1)c_H\left(-c c_\ell w^b_y\right) + \frac{c'_i(\bar{c}_i+1)}{2}w^b_y\\
&\leq \left(-\sigma_i^2 c c_H c_\ell c''_i(\bar{c}_i+1) + \frac{c'_i(\bar{c}_i+1)}{2}\right)w^b_y\\
& \leq 0.
\end{align}
Moreover,
\begin{align}
w^b_{x_i} + \frac{c'_i(\bar{c}_i+1)}{2}w^b_y &= w^b_{x_i} + \frac{c'_i(\bar{c}_i+1)}{2}w^b_y\\
&\leq-c c_\ell w^b_y + \frac{c'_i(\bar{c}_i+1)}{2}w^b_y\\
&= \left(-c c_\ell + \frac{c'_i(\bar{c}_i+1)}{2}\right)w^b_y\\
&< 0.
\end{align}
Thus, together with $w^b_{yy} \leq 0$, we have
$$
\psi_i(\bar{c}_i+1) = \sigma_i^2 c''_i(\bar{c}_i+1) w^b_{x_iy} + \sigma_i^2 c''_i(\bar{c}_i+1) c'_i(\bar{c}_i+1) w^b_{yy} + w^b_{x_i} + c'_i(\bar{c}_i+1) w^b_y < 0.
$$
Since \(\psi_i(0)= w_{x_i}^b>0\), the existence follows immediately from the Intermediate Value Theorem.

To show the uniform boundedness of $\lambda_f$, on the one hand, suppose there exists $ x_0 \in \mathbb{R}^d, y_0 \in \mathbb{R}, t_0 \in [0,T]$ such that $\lambda_{f,i}(x_0,y_0,t_0) > \bar{c}_i$, 
Then we have
$\frac{c_i'(\lambda_{f,i})}{c_i''(\lambda_{f,i})} \geq 2 \sigma_i^2 c c_H c_\ell $ and $c_i'(\lambda_{f,i}) > 2 c c_\ell$.
A similar argument yields
$$
\psi_i(\lambda_{f,i}) = \sigma_i^2 c''_i(\lambda_{f,i}) w^b_{x_iy} + \sigma_i^2 c''_i(\lambda_{f,i}) c'_i(\lambda_{f,i}) w^b_{yy} + w^b_{x_i} + c'_i(\lambda_{f,i}) w^b_y < 0,
$$
which leads to a contradiction.

On the other hand, suppose there exists $ x_0 \in \mathbb{R}^d, y_0 \in \mathbb{R}, t_0 \in [0,T]$ such that $\lambda_{f,i}(x_0,y_0,t_0) < \frac{1}{\hat{c}_i}$, where $\hat{c}_i$ is the minimum to satisfy these 2 conditions:
$\sigma_i^2 \big(1 + c_i'(\frac{1}{\hat{c}_i})\big) c''_i(\frac{1}{\hat{c}_i}) \leq \frac{1}{2 c_H}$ and $c'_i(\frac{1}{\hat{c}_i}) \leq \frac{1}{2c c_\ell}$. Notice that $\hat{c}_i$ is independent of $c_\lambda$ and $\epsilon$. Then we have $\sigma_i^2 \big(1 + c_i'(\lambda_{f,i})\big) c''_i(\lambda_{f,i}) < \frac{1}{2 c_H}$ and $c'_i(\lambda_{f,i}) < \frac{1}{2c c_\ell}$.
It follows that
\begin{align}
\sigma_i^2 c''_i(\lambda_{f,i})w^b_{x_iy} + \sigma_i^2 c''_i(\lambda_{f,i})c'_i(\lambda_{f,i})w^b_{yy} &\geq -\sigma_i^2 c''_i(\lambda_{f,i}) c_H w^b_{x_i} - \sigma_i^2 c''_i(\lambda_{f,i}) c'_i(\lambda_{f,i}) c_H w^b_{x_i}\\
&= - \left(1 + c'_i(\lambda_{f,i}) \right)  c''_i(\lambda_{f,i}) c_H \sigma_i^2 w^b_{x_i}\\
&> - \frac{w^b_{x_i}}{2}.
\end{align}
Additionally, 
we have
\begin{align}
\frac{w^b_{x_i}}{2} + c'_i(\lambda_{f,i}) w^b_y \geq \left(\frac{1}{2} -c c_\ell c'_i(\lambda_{f,i}) \right) w^b_{x_i} > 0.
\end{align}
Thus, 
$$
\psi_i(\lambda_{f,i}) = \sigma_i^2 c''_i(\lambda_{f,i})w^b_{x_iy} + \sigma_i^2 c''_i(\lambda_{f,i}) c'_i(\lambda_{f,i}) w^b_{yy} + w^b_{x_i} + c'_i(\lambda_{f,i}) w^b_y > 0,
$$
which leads to a contradiction.

Therefore, taking $\tilde{c}_i = \max\{\bar{c}_i,\hat{c}_i\}$, for every $c_\lambda > 1, \, 0 < \epsilon < 1$, we obtain
$$
\frac{1}{\tilde{c}_i} \leq \lambda_{f,i}(x,y,t) \leq \tilde{c}_i, \quad \forall x \in \mathbb{R}^d, y \in \mathbb{R}, t \in [0,T].
$$

When $c_\lambda < \tilde{c}_i$, $ \frac{1}{\tilde{c}_i} < \frac{1}{c_\lambda} \leq \lambda^*_{b,i} \leq c_\lambda < \tilde{c}_i$; when $c_\lambda \geq \tilde{c}_i$, $\frac{1}{\tilde{c}_i} < \lambda^*_{b,i} = \lambda_{f,i} < \tilde{c}_i$. The uniqueness and positivity of $\lambda^*_b$ follow immediately from Step 3.

Taking $\tilde{c}_\lambda := \max_{1 \leq i \leq d} \tilde{c}_i$, we conclude that $\lambda_f \in [\frac{1}{\tilde{c}_\lambda},\tilde{c}_\lambda]^d$ and $\lambda^*_b \in [\frac{1}{\tilde{c}_\lambda},\tilde{c}_\lambda]^d$.

\textbf{Proof of (2):}
Fix $(x_0,y_0) \in \mathbb{R}^d \times \mathbb{R}, \ R > 0$, and consider the cylinder $Q = B_R(x_0,y_0) \times (0,T)$ where $B_R(x_0,y_0) := \{(x,y)|\|x-x_0\|^2+(y-y_0)^2 < R^2\}$.
Given the growth condition in PDE \eqref{pde_w_bdd_lamb}, we consider the function
$$
\phi(x,y,t) := c_\phi \left( \prod_{j=1}^d \left(e^{2D x_j} + e^{-2D x_j}\right) \right) \left(e^{2D y} + e^{-2D y}\right) e^{\beta(T-t)},
$$
where $c_\phi > 0$ is chosen sufficiently large, independently of $\epsilon$ and $c_\lambda$, so that $|w^b(x,y,T)| \leq \phi(x,y,T)$. By part (1), when $c_\lambda>\tilde{c}_\lambda$, the cutoff
$c_\lambda$ in \eqref{pde_w_bdd_lamb} can be replaced by
$\tilde{c}_\lambda$ without changing $w^b$, while when
$c_\lambda\leq\tilde{c}_\lambda$, the control set is already
uniformly bounded. Therefore, together with $\epsilon\in(0,1)$, it is easy to verify that, by choosing $\beta>0$ sufficiently large, $\phi$ and $-\phi$ are respectively a supersolution and a subsolution of the PDE \eqref{pde_w_bdd_lamb} with $c_\lambda = \tilde{c}_\lambda$ whenever $c_\lambda>\tilde{c}_\lambda$. Moreover, $\beta$ can be chosen independently of $c_\lambda$ and $\epsilon$, which is crucial for deriving the uniform bound below.
Furthermore, noting that $-\phi(x,y,T) \leq w^b(x,y,T) \leq \phi(x,y,T)$, it follows from the comparison principle for the corresponding PDE \eqref{pde_w_bdd_lamb} that 
$$
-\phi(x,y,t) \leq w^b(x,y,t) \leq \phi(x,y,t), \quad \forall (x,y,t) \in \mathbb{R}^d \times \mathbb{R} \times [0,T].
$$
Consequently, $w^b$ is bounded in $\overline{Q} = \overline{B_R(x_0,y_0)} \times [0,T]$. Additionally, this local bound is independent of $\epsilon$ and $c_\lambda$ because $\beta$ and $c_\phi$ are independent of $\epsilon$ and $c_\lambda$.

For the partial derivatives, it follows from Lemma \ref{lem_viscosity_partial_derivatives} that $p^-_{j,\Delta}(x,y,t)$ are sub-solutions, while $p^+_{j,\Delta}(x,y,t)$ are super-solutions of the PDE \eqref{linear_pde_w^b}. 
Additionally, given the growth condition in PDE \eqref{linear_pde_w^b} and $0 < \Delta < 1$, we may increase $c_\phi$, if necessary, so that $p^-_{j,\Delta}(x,y,T) \leq \phi(x,y,T)$, while $p^+_{j,\Delta}(x,y,T) > 0$, for every $j\in\{1,2,\ldots,d\}$. 
Therefore, the comparison principle 
for PDE \eqref{linear_pde_w^b} (see Remark \ref{rmk_CP}), together with the concavity of $w^b(x,y,t)$, yields
\begin{align}
0 < p^+_{j,\Delta}(x,y,t) \leq p^-_{j,\Delta}(x,y,t) \leq \phi(x,y,t), \quad \forall (x,y,t) \in \mathbb{R}^d \times \mathbb{R} \times [0,T], \quad \forall j \in \{1,2,\dots,d\}.
\end{align}
Consequently, $p^+_{j,\Delta}(x,y,t)$ and $p^-_{j,\Delta}(x,y,t)$ are both bounded in $\overline{Q}$ and the bound is independent of $\epsilon$ and $c_\lambda$. Additionally, it follows from $w^b_{x_j}(x,y,t) = \lim_{\Delta \rightarrow 0}p^+_{j,\Delta}(x,y,t) = \lim_{\Delta \rightarrow 0}p^-_{j,\Delta}(x,y,t)$ that $w^b_{x_j}(x,y,t)$ is bounded in $\overline{Q}$, where the bound is independent of $\epsilon$ and $c_\lambda$. Notice that this argument works for any integer $j$ satisfying $1 \leq j \leq d$. Therefore, $w^b_{x}(x,y,t)$ is bounded in $\overline{Q}$.
The same argument yields the boundedness of $w^b_y(x,y,t)$ and the lower bounds for $w^b_{xx}(x,y,t)$ and $w^b_{yy}(x,y,t)$, while their upper bounds follow from the concavity of $w^b$. In particular, $w^b_{x_jx_j}\leq 0$ for any integer $j$ satisfying $1 \leq j \leq d$ and $w^b_{yy}\leq 0$. Moreover, since the Hessian of
$w^b$ is negative semidefinite,
\[
|w_{x_jx_k}^b|^2
\leq w_{x_jx_j}^b w_{x_kx_k}^b,
\qquad
|w_{x_jy}^b|^2
\leq w_{x_jx_j}^b w_{yy}^b.
\]
Hence, all components of $w_{xx}^b$ and $w_{xy}^b$ are locally uniformly bounded.

Consequently, the boundedness of $w^b_t$ in $\overline{Q}$ follows immediately from the boundedness of the above terms and the PDE \eqref{pde_w_bdd_lamb}. Hence, we prove that $\|w^b\|_{W^{2,1}_\infty(\overline{Q})} \leq C_Q$ for some $C_Q > 0$ that is independent of $\epsilon$ and $c_\lambda$. We conclude that $w^b \in W^{2,1}_{\infty,loc}(\mathbb{R}^d \times \mathbb{R} \times [0,T])$ because the argument applies for any $(x_0,y_0) \in \mathbb{R}^d \times \mathbb{R}$ and $R > 0$.
\end{proof}

\section[Regularity of w-epsilon]{Regularity of $w^\epsilon(x,y,t)$} \label{sec_cvg_w^b}
We now lift the restriction on the control set and return to the admissible set $\mathcal{A}(x, y, t)$. Our goal is to show that the corresponding value function inherits the regularity of $w^b$.
To this end, we consider the stochastic differential equation \eqref{auxiliary} and set
\begin{equation} \label{objective}
    w^\epsilon(x, y, t) = \sup _{\Lambda\in \mathcal{A}(x, y, t)} \mathbb{E}\left[U_P\big(\ell\left(x^{\lambda}(T)\right)-U_A^{-1}(y^{\lambda}(T))\big)\right], \forall x\in \mathbb{R}^d, y\in \mathbb{R}, t \in [0, T].
\end{equation}
Notice that Lemma \ref{lemma_bdd_obj} also works for $w^\epsilon(x,y,t)$. Together with the dynamic programming principle, $w^\epsilon(x, y, t)$ is the viscosity solution of the following PDE
\begin{equation} \label{pde_w_epsilon}
    \begin{cases}
        w_t + \sup_{\lambda \in [0,\infty)^d}
        \Bigg\{
        \frac{1+\epsilon}{2} \Tr[\sigma \sigma^\top w_{xx}] + [\sigma\sigma^\top\nabla c(\lambda)]^\top w_{xy} + \frac{1}{2} \|\sigma^\top\nabla c(\lambda)\|^2 w_{yy} + \lambda^\top w_x + c(\lambda) w_y \Bigg\} = 0,\\
        w(x, y,T) = U_P(\ell\left(x\right) - U_A^{-1}(y)),\\
        \sup\limits_{t \in [0,T], x \in \mathbb{R}^d, y \in \mathbb{R}} \frac{|w(x,y,t)|}{e^{D(\|x\|+|y|)}} < \infty, \text{ for some constant } D > 0.
    \end{cases}
\end{equation}
To derive the optimal $\lambda^*$, we have the following first-order condition:
    \begin{equation} \label{FOC}
        \nabla^2 c(\lambda^*) \sigma \sigma^\top  w^\epsilon_{xy} + \nabla^2 c(\lambda^*) \sigma \sigma^\top \nabla c(\lambda^*)w^\epsilon_{yy} + w^\epsilon_x + \nabla c(\lambda^*) w^\epsilon_y = 0.
    \end{equation}

To show that $w^\epsilon$ inherits the desired regularity, our argument proceeds in two steps. First, we establish the convergence of $w^b$ to $w^\epsilon$ as $c_\lambda \to \infty$ from the perspective of stochastic control. Second, we demonstrate that the limiting function inherits the regularity of $w^b$. The intuition for the second step is that Proposition~\ref{prop_unique_lambda_w^b} guarantees a uniform bound on the optimal control $\lambda^*_b$ that is independent of $c_\lambda$. Consequently, for sufficiently large $c_\lambda$, the function $w^b$ stabilizes, ensuring that the limit $w^\epsilon$ preserves the same regularity.

To this end, we define $\mathcal{A}^b_k(x, y, t) := \left\{ \Lambda \in \mathcal{A}(x, y, t) : \frac{1}{k} \leq \lambda_i(s) \leq k, \, \forall i = 1, \dots, d, \, \forall s \in [t, T] \right\}$ where $k > 1$ is an integer. Define the corresponding value function
\begin{align*}
    w_k^b(x, y, t)&=\sup _{\Lambda\in \mathcal{A}_k^b(x, y, t)} \mathbb{E}\left[U_P\big(\ell \left(x^{\lambda}(T)\right)-U_A^{-1}(y^{\lambda}(T))\big)\right]. 
    \end{align*}
By Theorem \ref{thm_regularity_w^b}, we have $w_k^b(x,y,t) \in C^{2+\alpha,\frac{2+\alpha}{2}}(\mathbb{R}^d \times \mathbb{R} \times [0,T])$. 
The desired regularity of $w^\epsilon$ is then established by the following theorem:
\begin{theorem} \label{thm_cvg_w_i_w}
    (i) $\lim_{k \rightarrow \infty}w_k^b(x,y,t) = w^\epsilon(x,y,t), \quad \forall (x,y,t) \in \mathbb{R}^d \times \mathbb{R} \times [0,T]$.\\
    (ii) $w^\epsilon(x,y,t) \in C^{2+\alpha,\frac{2+\alpha}{2}}\bigg(\mathbb{R}^d \times \mathbb{R} \times [0,T]\bigg)$ and $\lambda^*(x,y,t) \in [\frac{1}{\tilde{c}_\lambda},\tilde{c}_\lambda]^d, \quad \forall (x,y,t) \in \mathbb{R}^d \times \mathbb{R} \times [0,T]$.\\
        {(iii) $w^\epsilon \in W^{2,1}_{\infty,loc}(\mathbb{R}^d \times \mathbb{R} \times [0,T])$. In particular, for every $(x_0,y_0) \in \mathbb{R}^d \times \mathbb{R}$ and $ R > 0$, there exists a constant $C_Q > 0$, independent of $\epsilon$, such that $\|w^\epsilon\|_{W^{2,1}_\infty(\overline{Q})} \leq C_Q$, where $Q = B_R(x_0,y_0) \times (0,T)$ is the cylinder and $B_R(x_0,y_0) := \{(x,y)|\|x-x_0\|^2+(y-y_0)^2 < R^2\}$.}

    \end{theorem}

\section[Regularity of w-zero]{Regularity of $w^0(x,y,t)$} \label{sec_cvg_w}
We now turn to the setting without the noise ($\epsilon = 0$) to establish the regularity of the value function $w^0$. Our argument proceeds in two steps.
First, we show that $\lim_{\epsilon \to 0} w^\epsilon(x,y,t) = w^0$. Moreover, $w^0(x,y,t) \in W^{2,1}_{p,loc}(\mathbb{R}^d \times \mathbb{R} \times [0,T])$. To establish the pointwise limit, the upper bound $\lim_{\epsilon \to 0} w^\epsilon(x,y,t) \leq w^0$ follows directly from the comparison principle for the counterpart of PDE \eqref{pde_w_bdd_lamb} with $\epsilon=0$. Conversely, establishing the lower bound $\lim_{\epsilon \to 0} w^\epsilon(x,y,t) \geq w^0$ requires approximating the principal's utility $U_P$ via the modified asymptotically linear utility $U_P^M$ introduced in the proof of Theorem \ref{thm_cvg_w_i_w}. Second, we demonstrate that $w^0 \in W^{2,1}_{\infty,loc}(\mathbb{R}^d \times \mathbb{R} \times [0,T])$.
\begin{theorem} \label{thm_regularity_w^0}
        $w^0 \in W^{2,1}_{\infty,loc}(\mathbb{R}^d \times \mathbb{R} \times [0,T])$. That is, for any point $(x_0,y_0) \in \mathbb{R}^d \times \mathbb{R}$ and any radius $ R > 0$, we have $w^0 \in {W^{2,1}_\infty(\overline{Q})}$, where $Q = B_R(x_0,y_0) \times (0,T)$ is the cylinder and $B_R(x_0,y_0) := \{(x,y)|\|x-x_0\|^2+(y-y_0)^2 < R^2\}$.
\end{theorem}

\section{Conclusion} \label{sec_conclusion}
This paper establishes the local \(W^{2,1}_{\infty, loc}\) regularity of the value function associated with a multidimensional continuous-time principal-agent model with separable nonquadratic effort costs. The main difficulty is that the corresponding HJB equation is both degenerate and fully nonlinear, with coefficients that may be unbounded. We address this problem through a family of regularized control problems with bounded effort sets. The key step is to show that the optimal effort is unique, positive, and uniformly bounded independently of both the control restriction and the regularization parameter. This bound allows the control restriction to be removed and yields local derivative estimates that remain uniform as the additional noise vanishes. By combining the convergence of the regularized value function with the uniform estimates, we obtain the regularity of the original value function. The resulting regularity is weaker than the regularity available in the quadratic-cost setting, but it applies to a broader class of separable effort costs for which the quadratic-cost arguments are no longer available. Extending the analysis to nonseparable costs and more general volatility structures remains an open direction.

\bibliographystyle{plain}
\bibliography{PA_260424}

\clearpage

\appendix
\section*{Appendices}
\section{Proof of Lemma \ref{lemma_bdd_obj}}
\begin{proof}
    On the one hand, $w^0(x, y, t) \leq 0$ since $U_P \leq 0$ by Assumption \ref{assumptionUP}. On the other hand, taking $\lambda^0 = \mathbf{1} \in \mathcal{A}^0(x, y, t)$, we have
    \begin{align*}
        \hat{x}^{\lambda^0}(T) &= x + \mathbf{1} (T-t) + \sigma(\mathcal{B}(T) - \mathcal{B}(t)) \sim \mathcal{N}\bigg( x + \mathbf{1} (T-t),\sigma \sigma^\top (T-t) \bigg),\\
        y^{\lambda^0}(T) &= y + c(\mathbf{1})(T-t) + \left[\nabla c(\mathbf{1})\right]^\top \sigma (\mathcal{B}(T) - \mathcal{B}(t)) \sim \mathcal{N}\bigg( y + c(\mathbf{1})(T-t), \|\sigma^\top \nabla c(\mathbf{1})\|^2(T-t) \bigg).\\
    \end{align*}
    We aim to show $\mathbb{E}\left[e^{2c_Pc_\ell\|\hat{x}^{\lambda^0}(T)\|}\right], \mathbb{E}\left[e^{2c_P c \cdot y^{\lambda^0}(T) }\right], \mathbb{E}\left[e^{-2c_P c \cdot y^{\lambda^0}(T) }\right]$ are all uniformly bounded above using the moment-generating functions of normal random variables. Indeed,
                    \begin{align}
        \mathbb{E}\left[e^{2c_P c \cdot y^{\lambda^0}(T) }\right] &= e^{2c_P c \big(y + c(\mathbf{1})(T-t)\big) + \frac{4(c_P c)^2 \|\sigma^\top \nabla c(\mathbf{1})\|^2(T-t)}{2}}\\
        &\leq e^{2c_P c\big(y_0+R+c(\mathbf{1})T\big) + \frac{4(c_P c)^2 \|\sigma^\top \nabla c(\mathbf{1})\|^2 T}{2}}.
    \end{align}
    The same argument gives a uniform upper bound for $\mathbb{E}\left[e^{-2c_P c \cdot y^{\lambda^0}(T) }\right]$.
    Moreover, note that 
    $$
    \|\hat{x}^{\lambda^0}(T)\| = \sqrt{\sum_{i=1}^d \left(\hat{x}^{\lambda^0}(T)\right)_i^2} \leq \sqrt{d \cdot \max_{1 \le i \le d} \left(\hat{x}^{\lambda^0}(T)\right)_i^2} = \sqrt{d} \max_{1 \le i \le d} \left|\left(\hat{x}^{\lambda^0}(T)\right)_i\right|.
    $$
    Then it follows that
    \begin{align}
        \mathbb{E}\left[e^{2c_Pc_\ell\|\hat{x}^{\lambda^0}(T)\|}\right] &\leq \mathbb{E}\left[e^{2c_Pc_\ell \sqrt{d} \max_{1 \le i \le d} \left|\left(\hat{x}^{\lambda^0}(T)\right)_i\right|}\right]\\
        &= \mathbb{E}\left[\max_{1 \le i \le d} e^{2c_Pc_\ell \sqrt{d} \left|\left(\hat{x}^{\lambda^0}(T)\right)_i\right|}\right]\\
        &\leq \sum_{i=1}^d \mathbb{E}\left[e^{2c_Pc_\ell \sqrt{d} \left|\left(\hat{x}^{\lambda^0}(T)\right)_i\right|}\right]\\
        &\leq \sum_{i=1}^d \left(\mathbb{E}\left[e^{2c_Pc_\ell \sqrt{d} \left(\hat{x}^{\lambda^0}(T)\right)_i}\right] + \mathbb{E}\left[e^{-2c_Pc_\ell \sqrt{d} \left(\hat{x}^{\lambda^0}(T)\right)_i}\right]\right).
    \end{align}
    Since $\left(\hat{x}^{\lambda^0}(T)\right)_i \sim \mathcal{N}\left(x_i + T-t, \sigma_i^2 (T-t)\right)$, the same argument used for $\mathbb{E}\left[e^{2c_P c \cdot y^{\lambda^0}(T) }\right]$ applies here.

    It follows from Assumption \ref{assumptionUA} and the Mean Value Theorem that there exists $\xi_3$ between $0$ and $y^{\lambda^0}(T)$ such that $U_A^{-1}(y^{\lambda^0}(T)) = (U_A^{-1})'(\xi_3)y^{\lambda^0}(T) + U_A^{-1}(0)$.
    Thus, by the Cauchy--Schwarz inequality, we have
    \begin{align*}
        w^0(x, y, t) &= \sup _{\Lambda\in \mathcal{A}^0(x, y, t)} \mathbb{E}\left[U_P\big(\ell(\hat{x}^{\lambda}(T))-U_A^{-1}(y^{\lambda}(T))\big)\right]\\
        &\geq \mathbb{E}\left[U_P\big(\ell(\hat{x}^{\lambda^0}(T))-U_A^{-1}(y^{\lambda^0}(T))\big)\right]\\
        &= \mathbb{E}\left[U_P\big(\ell(\hat{x}^{\lambda^0}(T)) - (U_A^{-1})'(\xi_3)y^{\lambda^0}(T) - U_A^{-1}(0)\big)\right]\\
        & \geq \mathbb{E}\left[-C_Pe^{-c_P\big(\ell(\hat{x}^{\lambda^0}(T)) - (U_A^{-1})'(\xi_3)y^{\lambda^0}(T) - U_A^{-1}(0)\big)} - C_P\right]\\
        & \geq \mathbb{E}\left[-C_Pe^{c_P\left|\ell(\hat{x}^{\lambda^0}(T)) - (U_A^{-1})'(\xi_3)y^{\lambda^0}(T) - U_A^{-1}(0)\right|}\right]- C_P\\
        & \geq \mathbb{E}\left[-C_Pe^{c_P\left|\ell(\hat{x}^{\lambda^0}(T))\right|+ c_P\left|U_A^{-1}(0)\right| + c_P\left|(U_A^{-1})'(\xi_3)y^{\lambda^0}(T)\right| }\right]- C_P\\
        & \geq \mathbb{E}\left[-C_Pe^{c_P\left|U_A^{-1}(0)\right|}e^{c_P\left|\ell(\hat{x}^{\lambda^0}(T))\right|+ c_P c \cdot\left|y^{\lambda^0}(T)\right| }\right]- C_P\\
        & \geq -C_Pe^{c_P\left(\left|U_A^{-1}(0)\right|+c_\ell \right)} \mathbb{E}\left[e^{c_Pc_\ell \|\hat{x}^{\lambda^0}(T)\|+ c_P c \cdot\left|y^{\lambda^0}(T)\right| }\right]- C_P\\
        & \geq -C_Pe^{c_P\left(\left|U_A^{-1}(0)\right|+c_\ell \right)}\bigg(\mathbb{E}\left[e^{2c_Pc_\ell\|\hat{x}^{\lambda^0}(T)\|}\right]\bigg)^{\frac{1}{2}} \bigg(\mathbb{E}\left[e^{2c_P c \cdot\left|y^{\lambda^0}(T)\right| }\right]\bigg)^{\frac{1}{2}}- C_P\\
        & \geq -C_P\bigg(\mathbb{E}\left[e^{2c_Pc_\ell\|\hat{x}^{\lambda^0}(T)\|}\right]\bigg)^{\frac{1}{2}} \cdot \bigg(\mathbb{E}\left[e^{2c_P c \cdot y^{\lambda^0}(T) }\right] + \mathbb{E}\left[e^{-2c_P c \cdot y^{\lambda^0}(T) }\right]\bigg)^{\frac{1}{2}} - C_P,
    \end{align*}
    for some constant $C_P > 0$.
    Hence, $w^0(x,y,t)$ is bounded below in $\overline{Q}$.
\end{proof}

\section{Proof of Theorem \ref{thm_regularity_w^b}}
\begin{proof}
    We use Theorem 6.4.4 of \cite{krylov1987nonlinear} to establish the regularity of $w^b(x,y,t)$. To use that theorem, we need the convexity of the PDE operator. To show this, we define
\begin{align*}
    L(D^2w, p, \lambda) &=\frac{1+\epsilon}{2} \Tr[\sigma \sigma^\top w_{xx}] + [\sigma \sigma^\top \nabla c(\lambda)]^\top w_{xy} + \frac{1}{2} \|\sigma^\top \nabla c(\lambda)\|^2 w_{yy} + \lambda^\top p_1 + c(\lambda) p_2,\\
    F(D^2w, p) &= \sup_{\lambda \in [\frac{1}{c_\lambda},c_\lambda]^d} L(D^2w, p, \lambda),
\end{align*}
where $p = (p_1,p_2) \in \mathbb{R}^d \times \mathbb{R}$. The convexity of $F$ is guaranteed by
\begin{lemma} \label{lemma_convex_pde}
    $F(D^2w,p)$ is convex with respect to $D^2w$.
\end{lemma}
\begin{proof}[Proof of Lemma \ref{lemma_convex_pde}]
    Let $\mathcal{S}(\mathbb{R})$ denote the set of $(d+1) \times (d+1)$ symmetric matrices with real entries. Fix $p \in \mathbb{R}^{d+1}$. For any $D^2 w_1, D^2 w_2 \in \mathcal{S}(\mathbb{R})$ and any $\lambda_0 \in [\frac{1}{c_\lambda},c_\lambda]^d$, we have
    \begin{align*}
        F(D^2w_1,p) &= \sup_{\lambda \in [\frac{1}{c_\lambda},c_\lambda]^d} L(D^2w_1, p, \lambda) \geq L(D^2w_1, p, \lambda_0),\\
        F(D^2w_2,p) &= \sup_{\lambda \in [\frac{1}{c_\lambda},c_\lambda]^d} L(D^2w_2, p, \lambda) \geq L(D^2w_2, p, \lambda_0).
    \end{align*}
    Then it follows from the linearity of $L$ that
    \begin{align*}
        \frac{F(D^2w_1,p)+F(D^2w_2,p)}{2} \geq \frac{L(D^2w_1, p, \lambda_0)+L(D^2w_2, p, \lambda_0)}{2} = L(\frac{D^2w_1+D^2w_2}{2}, p, \lambda_0).
    \end{align*}
    Since this holds for any $\lambda_0 \in [\frac{1}{c_\lambda},c_\lambda]^d$, taking supremum over $\lambda_0$ yields
    \begin{align*}
        \frac{F(D^2w_1, p)+F(D^2w_2, p)}{2} \geq F(\frac{D^2w_1+D^2w_2}{2}, p).
    \end{align*}
\end{proof}

    The above lemma will be used to prove the theorem. Precisely, fix $(x_0,y_0) \in \mathbb{R}^d \times \mathbb{R}$ and $R > 0$, and consider the cylinder $Q = B_R(x_0,y_0) \times (0,T)$ where $B_R(x_0,y_0) := \{(x,y)|\|x-x_0\|^2+(y-y_0)^2 < R^2\}$. This cylinder satisfies the exterior sphere condition since this domain has a $C^2$ boundary. Then we consider the PDE problem on this cylinder. Moreover, since $w^b(x, y, t)$ is the viscosity solution of the PDE \eqref{pde_w_bdd_lamb}, we assign this value as the lateral boundary condition, i.e., we consider
    \begin{equation} \label{pde_w_bdd_lamb_local}
    \begin{cases}
        w_t + \sup_{\lambda \in [\frac{1}{c_\lambda},c_\lambda]^d}
        \Bigg\{
        \frac{1+\epsilon}{2} \Tr[\sigma \sigma^\top w_{xx}] + [\sigma \sigma^\top \nabla c(\lambda)]^\top w_{xy}\\
        \qquad \qquad \qquad \qquad + \frac{1}{2} \|\sigma^\top \nabla c(\lambda)\|^2 w_{yy} + \lambda^\top w_x + c(\lambda) w_y \Bigg\} = 0, &(x,y,t) \in B_R(x_0,y_0) \times (0,T),\\
        w(x,y,T) = U_P(\ell(x) - U_A^{-1}(y)),   &(x,y) \in \overline{B_R(x_0,y_0)}, \\
        w(x,y,t) = w^b(x, y, t), \quad &(x,y,t) \in \partial B_R(x_0,y_0) \times (0,T).
    \end{cases}
    \end{equation}
    It follows that $w^b(x, y, t)$ is the viscosity solution of the PDE \eqref{pde_w_bdd_lamb_local}.
    
        We next verify that the terminal and lateral boundary conditions in \eqref{pde_w_bdd_lamb_local} are continuous and bounded. On the one hand, by the regularity of utility functions, the function $U_P(\ell(x) - U_A^{-1}(y))$ is continuous on $\overline{Q}$. Furthermore, let $(w^b)^*$ and $(w^b)_*$ denote the upper- and lower-semicontinuous envelopes of $w^b$, respectively. The dynamic programming principle implies that $(w^b)^*$ and $(w^b)_*$ are respectively a viscosity subsolution and a viscosity supersolution of \eqref{pde_w_bdd_lamb}. Hence, the comparison principle for PDE \eqref{pde_w_bdd_lamb} (see Remark \ref{rmk_CP}) yields $(w^b)^*\leq (w^b)_*$. Since the reverse inequality holds by definition, we obtain $(w^b)^*=(w^b)_*$, and therefore
        $w^b$ is continuous on $\overline{Q}$.
        On the other hand, by Lemma \ref{lemma_bdd_obj} and the regularity of utility functions, the two functions above are uniformly bounded on the cylinder $Q$. Moreover, we have $U_P(\ell(x) - U_A^{-1}(y)) \in C^{2+\alpha^0}(B_R(x_0,y_0))$. By Lemma \ref{lemma_convex_pde} and Theorem 6.4.4 of \cite{krylov1987nonlinear}, the PDE \eqref{pde_w_bdd_lamb_local} admits a solution $w^l(x,y,t) \in C^{2+\alpha,\frac{2+\alpha}{2}}(\overline{B_R^0(x_0,y_0)} \times [0,T])$ where $B_R^0(x_0,y_0) \subset B_R(x_0,y_0)$. By the comparison principle for PDE \eqref{pde_w_bdd_lamb_local} (see Remark \ref{rmk_CP}),
        we have $w^b(x,y,t) = w^l(x,y,t) \in C^{2+\alpha,\frac{2+\alpha}{2}}(\overline{B_R^0(x_0,y_0)} \times [0,T])$. Moreover, $\|w^b\|_{C^{2+\alpha,\frac{2+\alpha}{2}}}$ is locally uniformly bounded. Since this is true for any $(x_0,y_0) \in \mathbb{R}^d \times \mathbb{R}$, we have $w^b(x,y,t) \in C_{loc}^{2+\alpha,\frac{2+\alpha}{2}}(\mathbb{R}^d \times \mathbb{R} \times [0,T])$.
\end{proof}

\section{Proof of Lemma \ref{lemma_sec_derivative_w^b}}
\begin{proof} [Proof of Lemma \ref{lemma_sec_derivative_w^b}]
     For $\hat{k} \in \{1,2\}$, given $x^{(\hat{k})} \in \mathbb{R}^d$, $y^{(\hat{k})} \in \mathbb{R}$, $t \in [0,T]$ and $c_\lambda$, we take any $\lambda^{(1)} \in \mathcal{A}^b(x^{(1)}, y^{(1)}, t)$, $\lambda^{(2)} \in \mathcal{A}^b(x^{(2)}, y^{(2)}, t)$. Consider the stochastic control problem:
    \begin{equation}
    \begin{cases}
    d x^{(\hat{k})}(s)= {\lambda^{(\hat{k})}}(s) d s+\sigma d \mathcal{B}(s) + \sqrt{\epsilon} \sigma d\mathcal{W}(s), \\
    d y^{(\hat{k})}(s)=c(\lambda^{(\hat{k})}(s)) d s+[\nabla c(\lambda^{(\hat{k})}(s))]^\top \sigma d \mathcal{B}(s), \\
    x^{(\hat{k})}(t)=x^{(\hat{k})}, \quad y^{(\hat{k})}(t)=y^{(\hat{k})}. 
    \end{cases}
    \end{equation}
    Let $\hat{\lambda}$ satisfy
    $$
    \nabla c(\hat{\lambda}) = \frac{\nabla c(\lambda^{(1)}) + \nabla c(\lambda^{(2)})}{2} \Rightarrow c'_i(\hat{\lambda}_i) = \frac{c'_i(\lambda^{(1)}_i) + c'_i(\lambda^{(2)}_i)}{2}, \quad 1 \leq i \leq d.
    $$
    Notice that since $c''_i > 0$, $c'_i$ is increasing. Thus, $c'_i(\hat{\lambda}_i)$ lies between $c'_i(\lambda^{(1)}_i)$ and $c'_i(\lambda^{(2)}_i)$. Hence, $\hat{\lambda}_i \in [\frac{1}{c_\lambda}, c_\lambda]$ and it follows that $\hat{\lambda} \in \mathcal{A}^b(\hat{x},\hat{y},t)$.
    
    We consider $\hat{x}(s)$ and $\hat{y}(s)$ satisfying
    \begin{equation}
    \begin{cases}
    d \hat{x}(s)= \hat{\lambda}(s) d s+\sigma d \mathcal{B}(s)+ \sqrt{\epsilon} \sigma d\mathcal{W}(s), \\
    d \hat{y}(s)=c(\hat{\lambda}(s)) d s+\nabla c(\hat{\lambda}(s))^\top \sigma d \mathcal{B}(s), \\
    \hat{x}(t)=\frac{x^{(1)}+x^{(2)}}{2}, \quad \hat{y}(t)=\frac{y^{(1)}+y^{(2)}}{2}. 
    \end{cases}
    \end{equation}
    On the one hand, since $c^{(3)}_i \geq 0$, it follows that $c'_i$ is convex, so Jensen's inequality gives
    \begin{align*}
        &\frac{c'_i(\lambda^{(1)}_i) + c'_i(\lambda^{(2)}_i)}{2} \geq c'_i\bigg(\frac{\lambda^{(1)}_i+\lambda^{(2)}_i}{2}\bigg)\\
        \Leftrightarrow \quad & (c'_i)^{-1}\bigg[\frac{c'_i(\lambda^{(1)}_i) + c'_i(\lambda^{(2)}_i)}{2}\bigg] \geq \frac{\lambda^{(1)}_i+\lambda^{(2)}_i}{2}  \quad ((c'_i)^{-1} \, \text{is increasing} )\\
        \Leftrightarrow \quad & \hat{\lambda}_i \geq 
        \frac{\lambda^{(1)}_i+\lambda^{(2)}_i}{2},
    \end{align*}
    which indicates $\hat{x}(T) \geq \frac{x^{(1)}(T) + x^{(2)}(T)}{2}$.
    On the other hand, by inequality \eqref{general_mean_ineq} from Assumption \ref{assumption_cost}, we have $c_i(\hat{\lambda}_i) \leq \frac{c_i(\lambda^{(1)}_i) + c_i(\lambda^{(2)}_i)}{2}$. Therefore, with the same initial condition $\frac{y^{(1)}+y^{(2)}}{2}$, we have $\hat{y}(T) \leq \frac{y^{(1)}(T)+y^{(2)}(T)}{2}$.
    
    By the monotonicity and concavity of $U_P$ and $\ell$, together with the monotonicity and convexity of $U_A^{-1}$, we have
    \begin{align*}
        &\quad w^b(\frac{1}{2}(x^{(1)}+x^{(2)}), \frac{1}{2}(y^{(1)}+y^{(2)}), t) \\
        &\geq
        \mathbb{E}\bigg[U_P\bigg(\ell(\hat{x}(T)) - U_A^{-1}\big(\hat{y}(T)\big)\bigg)\bigg]\\
        &\geq \mathbb{E}\bigg[U_P\bigg(\ell\left(\frac{x^{(1)}(T) + x^{(2)}(T)}{2}\right) - U_A^{-1}\big(\frac{y^{(1)}(T)+y^{(2)}(T)}{2}\big)\bigg)\bigg]\\
        &\geq \mathbb{E}\bigg[U_P\bigg(\frac{\ell\left(x^{(1)}(T)\right) + \ell\left(x^{(2)}(T)\right)}{2} -\frac{U_A^{-1}\big(y^{(1)}(T)\big)+U_A^{-1}\big(y^{(2)}(T)\big)}{2}\bigg)\bigg]\\
        &= \mathbb{E}\bigg[U_P\bigg(\frac{\bigg[\ell\left(x^{(1)}(T)\right) - U_A^{-1}\big(y^{(1)}(T)\big)\bigg] + \bigg[\ell\left(x^{(2)}(T)\right) - U_A^{-1}\big(y^{(2)}(T)\big)\bigg]}{2} \bigg)\bigg]\\
        &\geq \frac{1}{2} \mathbb{E}\bigg[U_P\bigg(\ell\left(x^{(1)}(T)\right) - U_A^{-1}\big(y^{(1)}(T)\big) \bigg)\bigg]\\
        &\ \ + \frac{1}{2} \mathbb{E}\bigg[U_P\bigg(\ell\left(x^{(2)}(T)\right) - U_A^{-1}\big(y^{(2)}(T)\big) \bigg)\bigg].
    \end{align*}
    Taking the supremum over $\lambda^{(1)}$ and $\lambda^{(2)}$, we have
    $$
    w^b(\frac{1}{2}(x^{(1)}+x^{(2)}), \frac{1}{2}(y^{(1)}+y^{(2)}), t) \geq \frac{w^b(x^{(1)},y^{(1)},t) + w^b(x^{(2)},y^{(2)},t)}{2},
    $$
    which establishes the concavity of $w^b(x,y,t)$ with respect to $(x,y)$. Hence, the Hessian matrix of $w^b$ with respect to $x$ and $y$,
    \[
    \nabla^2_{(x,y)} w^b = 
    \begin{pmatrix}
    w^b_{xx} & w^b_{xy} \\
    (w^b_{xy})^\top & w^b_{yy}
    \end{pmatrix}
    \]
    is negative semidefinite, which proves the first part of the lemma.
    
    Moreover, since every principal submatrix of a negative semidefinite matrix must also be negative semidefinite, it follows that $w^b_{yy} \leq 0$, $w^b_{x_j x_j} \leq 0$, $(w^b_{x_j y})^2 \leq w^b_{x_j x_j} w^b_{yy}$, and $(w^b_{x_j x_k})^2 \leq w^b_{x_j x_j} w^b_{x_kx_k}$ for $j,k \in \{1,2,\dots,d\}$ and $j \neq k$.
    
                            \end{proof}

\section{Proof of Lemma \ref{lem_viscosity_partial_derivatives}}
\begin{proof} [Proof of Lemma \ref{lem_viscosity_partial_derivatives}]
    \textbf{Step 1: We show that $p^+_{j,\Delta}(x,y,t)$, $p^-_{j,\Delta}(x,y,t)$, $q^+_\Delta(x,y,t)$, $q^-_\Delta(x,y,t)$, $s_{j,\Delta}(x,y,t)$, $r_\Delta(x,y,t)$ satisfy the growth condition in PDE \eqref{linear_pde_w^b}.}\\
    Define $f(x,y) = U_P\big(\ell(x)-U_A^{-1}(y)\big)$. Then $w^b(x,y,t) = \sup_{\Lambda \in \mathcal{A}^b} \mathbb{E}\big[f(x^\lambda(T), y^\lambda(T))\big]$ and
    \begin{align}
        f_{x_j} &= U_P'\big(\ell(x)-U_A^{-1}(y)\big) \ell_{x_j}(x) > 0, \\
        f_y &= -U_P'\big(\ell(x)-U_A^{-1}(y)\big) \cdot (U_A^{-1})'(y), \\
        f_{x_j x_j} &= U_P''\big(\ell(x)-U_A^{-1}(y)\big) \left(\ell_{x_j}(x)\right)^2 + U_P'\big(\ell(x)-U_A^{-1}(y)\big) \ell_{x_j x_j}(x) \leq 0, \\
        f_{yy} &= U_P''\big(\ell(x)-U_A^{-1}(y)\big)\cdot \big((U_A^{-1})'(y)\big)^2 - U_P'\big(\ell(x)-U_A^{-1}(y)\big)\cdot(U_A^{-1})''(y).
    \end{align}
    By Assumption \ref{assumptionUA}, Assumption \ref{assumptionUP} and Remark \ref{rmk_growth_U_P'}, we also have
    \begin{equation} \label{growth_partial_derivatives_terminal}
    |\frac{\partial^{\alpha+\beta} f}{\partial x_j^{\alpha} \partial y^{\beta}}| \leq K e^{\gamma(\|x\| + |y|)},
    \end{equation}
    for some real numbers $K,\gamma > 0$, and non-negative integers $\alpha,\beta$ satisfying $\alpha+\beta \in \{1, 2\}$.
        For any $\eta > 0$, by the definition of the supremum, there exists an $\eta$-optimal control $\Lambda^\eta = \{\lambda^\eta(s)\}_{s=t}^T \in \mathcal{A}^b(x,y,t)$ such that
    $$
    w^b(x,y,t) = \sup_{\Lambda \in \mathcal{A}^b(x,y,t)} \mathbb{E}\big[f(x^\lambda(T), y^\lambda(T))\big] \leq \mathbb{E}\big[f(x^{\lambda^\eta}(T), y^{\lambda^\eta}(T))\big] + \eta.
    $$
    Moreover, since $\Lambda^\eta \in \mathcal{A}^b(x+\Delta e_j,y,t)$, we have
    \begin{align}
        w^b(x+\Delta e_j, y, t) &= \sup_{\Lambda \in \mathcal{A}^b(x+\Delta e_j,y,t)} \mathbb{E}\bigg[f\big(x^\lambda(T), y^\lambda(T)\big)\bigg] \geq \mathbb{E}\bigg[f\big(x^{\lambda^\eta}(T)+\Delta e_j, y^{\lambda^\eta}(T)\big)\bigg].
    \end{align}
    It follows that
    \[
    p^+_{j,\Delta}(x,y,t) = \frac{w^b(x+\Delta e_j, y, t) - w^b(x, y, t)}{\Delta} \geq \mathbb{E}\Bigg[\frac{f\big(x^{\lambda^\eta}(T)+\Delta e_j, y^{\lambda^\eta}(T)\big) - f\big(x^{\lambda^\eta}(T), y^{\lambda^\eta}(T)\big)}{\Delta}\Bigg] - \frac{\eta}{\Delta}.
    \]
    By the Mean Value Theorem, the difference quotient equals $f_{x_j}(x^{\lambda^\eta}(T) + \xi_1 e_j, y^{\lambda^\eta}(T))$ for some $\xi_1 \in (0,\Delta)$. 
    
    Since $0 <\Delta < 1$, it follows from \eqref{growth_partial_derivatives_terminal} that
    $$
    |f_{x_j}(x^{\lambda^\eta}(T) + \xi_1 e_j, y^{\lambda^\eta}(T))| \leq K e^{\gamma(\|x^{\lambda^\eta}(T)\| + 1 + |y^{\lambda^\eta}(T)|)}.
    $$
    We claim that the right-hand side is integrable. To this end, we first consider $\mathbb{E}[e^{\|2\gamma x^{\lambda^\eta}(T)\|}]$. Notice that
    \begin{align}
        \|2\gamma x^{\lambda^\eta}(T)\| &= \left \| 2\gamma \int_t^T \lambda^\eta(s) ds + 2\gamma \sigma (\mathcal{B}(T)-\mathcal{B}(t)) + 2\gamma \sqrt{\epsilon} \sigma (\mathcal{W}(T)-\mathcal{W}(t)) + 2\gamma x\right\|\\
        &\leq 2\gamma \sqrt{d} c_\lambda (T-t) + \|2\gamma\sigma (\mathcal{B}(T)-\mathcal{B}(t))\| +  \|2\gamma \sqrt{\epsilon} \sigma (\mathcal{W}(T)-\mathcal{W}(t))\| + 2\gamma \|x\|,
    \end{align}
    and $2\gamma \sigma (\mathcal{B}(T)-\mathcal{B}(t)) \sim \mathcal{N}(0, 4 \gamma^2 \sigma \sigma^\top(T-t))$. Moreover, we have 
    \begin{align}
    \|2\gamma \sigma (\mathcal{B}(T)-\mathcal{B}(t))\| = \sqrt{\sum_{i=1}^d \left(2\gamma \sigma (\mathcal{B}(T)-\mathcal{B}(t))\right)_i^2} &\leq \sqrt{d \cdot \max_{1 \le i \le d} \left(2\gamma \sigma (\mathcal{B}(T)-\mathcal{B}(t))\right)_i^2}\\
    &= \sqrt{d} \max_{1 \le i \le d} \left|\left(2\gamma \sigma (\mathcal{B}(T)-\mathcal{B}(t))\right)_i\right|.
    \end{align}
    It follows that $\left(2\gamma \sigma (\mathcal{B}(T)-\mathcal{B}(t))\right)_i \sim \mathcal{N}\left(0, 4 \gamma^2 \sigma_i^2 (T-t)\right)$. Additionally, we have
    \begin{align}
        &\quad  \mathbb{E}[e^{\|2\gamma \sigma (\mathcal{B}(T)-\mathcal{B}(t))\|}] \leq \mathbb{E}[e^{ \sqrt{d} \max_{1 \le i \le d} \left|\left(2\gamma \sigma (\mathcal{B}(T)-\mathcal{B}(t))\right)_i\right|}] \\
        &= \mathbb{E}[ \max_{1 \le i \le d} e^{\sqrt{d}\left|\left(2\gamma \sigma (\mathcal{B}(T)-\mathcal{B}(t))\right)_i\right|}]\\
        &\leq \sum_{i=1}^d \mathbb{E}[ e^{\sqrt{d}\left|\left(2\gamma \sigma (\mathcal{B}(T)-\mathcal{B}(t))\right)_i\right|}]\\
        &= 2 \sum_{i=1}^d e^{{2 d \gamma^2 \sigma_i^2 (T-t)}} \Phi \left(2\sqrt{d}\gamma |\sigma_i| \sqrt{T-t}\right),
    \end{align}
    and
    \begin{align}
        \mathbb{E}[e^{\|2\gamma \sqrt{\epsilon} \sigma (\mathcal{W}(T)-\mathcal{W}(t))\|}] &\leq \mathbb{E}[e^{ \sqrt{d} \max_{1 \le i \le d} \left|\left(2\gamma \sqrt{\epsilon} \sigma (\mathcal{W}(T)-\mathcal{W}(t))\right)_i\right|}]\\
        &= \mathbb{E}[ \max_{1 \le i \le d} e^{\sqrt{d}\left|\left(2\gamma \sqrt{\epsilon} \sigma (\mathcal{W}(T)-\mathcal{W}(t))\right)_i\right|}]\\
        &\leq \sum_{i=1}^d \mathbb{E}[ e^{\sqrt{d}\left|\left(2\gamma \sqrt{\epsilon} \sigma (\mathcal{W}(T)-\mathcal{W}(t))\right)_i\right|}]\\
        &= 2 \sum_{i=1}^d e^{{2 d \gamma^2 \epsilon \sigma_i^2 (T-t)}} \Phi \left(2\sqrt{d}\gamma \sqrt{\epsilon} |\sigma_i| \sqrt{T-t}\right).
    \end{align}
    Therefore, noting the independence of $\mathcal{B}(T)-\mathcal{B}(t)$ and $\mathcal{W}(T)-\mathcal{W}(t)$, we have
    \begin{align}
        \mathbb{E}[e^{\|2\gamma x^{\lambda^\eta}(T)\|}] \leq e^{2\gamma \sqrt{d} c_\lambda (T-t)} \mathbb{E}[e^{\|2\gamma \sigma (\mathcal{B}(T)-\mathcal{B}(t))\|}]
        \mathbb{E}[e^{\|2\gamma \sqrt{\epsilon} \sigma (\mathcal{W}(T)-\mathcal{W}(t))\|}] e^{2\gamma \|x\|} =: \tilde{K} e^{2\gamma \|x\|}.
    \end{align}
    We next consider $\mathbb{E}[e^{ 2 \gamma |y^{\lambda^\eta}(T)|}]$. For $s \in [t,T]$, define the martingale $M_s = \int_t^s 2\gamma \nabla c(\lambda^\eta(r))^\top \sigma d\mathcal{B}(r)$ with quadratic variation $\langle M \rangle_s = \int_t^s \|2\gamma \sigma^\top \nabla c(\lambda^\eta(r))\|^2 dr \leq 4\gamma^2 \sum_{i=1}^d \sigma_i^2 (c'_i(c_\lambda))^2(s-t)$. Notice that $Z_s = e^{M_s - \frac{1}{2}\langle M \rangle_s}$ is also a martingale because the integrand is bounded. It follows that
    \begin{align}
        \mathbb{E}[e^{M_T}] = \mathbb{E}[Z_T \cdot e^{\frac{1}{2}\langle M \rangle_T}] &\leq \mathbb{E}[Z_T \cdot e^{2 \gamma^2 \sum_{i=1}^d \sigma_i^2 (c'_i(c_\lambda))^2(T-t)}] = \mathbb{E}[Z_T] \cdot e^{2 \gamma^2 \sum_{i=1}^d \sigma_i^2 (c'_i(c_\lambda))^2(T-t)} \\
        &= Z_t \cdot e^{2 \gamma^2 \sum_{i=1}^d \sigma_i^2 (c'_i(c_\lambda))^2 (T-t)} = e^{2 \gamma^2 \sum_{i=1}^d \sigma_i^2 (c'_i(c_\lambda))^2(T-t)}.
    \end{align}
    The same argument yields that $\mathbb{E}[e^{-M_T}]$ is also bounded above. It follows from $e^{|M_T|} \leq e^{M_T} + e^{-M_T}$ that 
    $$
    \mathbb{E}[e^{|M_T|}] \leq 2 \cdot e^{2 \gamma^2 \sum_{i=1}^d \sigma_i^2 (c'_i(c_\lambda))^2(T-t)}.
    $$
    Therefore,
    \begin{align}
        \mathbb{E}[e^{ 2\gamma |y^{\lambda^\eta}(T)|}] & \leq  e^{2\gamma c(c_\lambda \mathbf{1})(T-t)} \cdot e^{2\gamma |y|} \cdot \mathbb{E}[e^{ 2\gamma |\int_t^T \nabla c(\lambda^\eta(s))^\top \sigma d\mathcal{B}(s)|}] = e^{2\gamma c(c_\lambda \mathbf{1})(T-t)} \cdot e^{2\gamma |y|} \cdot \mathbb{E}[e^{ |M_T|}]\\
        &\leq e^{2\gamma c(c_\lambda \mathbf{1})(T-t)} \cdot e^{2\gamma |y|} \cdot 2 \cdot e^{2 \gamma^2 \sum_{i=1}^d \sigma_i^2 (c'_i(c_\lambda))^2(T-t)}.
    \end{align}
    It follows from the Cauchy--Schwarz inequality that
    \begin{align} \label{condition_DCT}
        \mathbb{E}[K e^{\gamma(\|x^{\lambda^\eta}(T)\| + 1 + |y^{\lambda^\eta}(T)|)}] &\leq K \cdot e^{\gamma} \cdot \mathbb{E}[e^{2\gamma \|x^{\lambda^\eta}(T)\|}]^{\frac{1}{2}} \cdot \mathbb{E}[e^{2\gamma |y^{\lambda^\eta}(T)|}]^{\frac{1}{2}}\\
                & \leq K \cdot e^{\gamma} \cdot \sqrt{\tilde{K}} \cdot e^{\gamma \|x\|} \cdot e^{\gamma c(c_\lambda \mathbf{1})(T-t)} \cdot e^{\gamma |y|} \cdot \sqrt{2} \cdot e^{ \gamma^2 \sum_{i=1}^d \sigma_i^2 (c'_i(c_\lambda))^2(T-t)}\\
        &=: \kappa e^{\gamma(\|x\|+|y|)} < \infty.
    \end{align}
    Therefore, we get
    \[
    \frac{w^b(x+\Delta e_j, y, t) - w^b(x, y, t)}{\Delta} \geq \mathbb{E}[-K e^{\gamma(\|x^{\lambda^\eta}(T)\| + 1 + |y^{\lambda^\eta}(T)|)}] - \frac{\eta}{\Delta} \geq -\kappa e^{\gamma(\|x\|+|y|)} - \frac{\eta}{\Delta}.
    \]
    Since $\kappa$ and $\gamma$ are independent of $\eta$, taking the limit as $\eta \rightarrow 0$ yields
    \[
    \frac{w^b(x+\Delta e_j, y, t) - w^b(x, y, t)}{\Delta} \geq -\kappa e^{\gamma(\|x\|+|y|)}.
    \]
    Similarly, there exists $\Lambda^+_b = \{\lambda^+_b(s)\}_{s=t}^T$ such that
    \begin{align}
        w^b(x+\Delta e_j, y, t) &= \sup_{\Lambda \in \mathcal{A}^b(x+\Delta e_j,y,t)} \mathbb{E}\bigg[f\big(x^\lambda(T), y^\lambda(T)\big)\bigg] \leq \mathbb{E}\bigg[f\big(x^{\lambda^+_b}(T)+\Delta e_j, y^{\lambda^+_b}(T)\big)\bigg] + \eta.
    \end{align}
    The same argument yields
    \begin{align}
    \frac{w^b(x+\Delta e_j, y, t) - w^b(x, y, t)}{\Delta} &\leq \mathbb{E}\Bigg[\frac{f\big(x^{\lambda^+_b}(T)+\Delta e_j, y^{\lambda^+_b}(T)\big) - f\big(x^{\lambda^+_b}(T), y^{\lambda^+_b}(T)\big)}{\Delta}\Bigg] + \frac{\eta}{\Delta}\\
    &= \mathbb{E}\Bigg[f_{x_j}(x^{\lambda^+_b}(T)+\hat{\xi}_1 e_j, y^{\lambda^+_b}(T))\Bigg] + \frac{\eta}{\Delta}\\
    &\leq \mathbb{E}\Bigg[K e^{\gamma(\|x^{\lambda^+_b}(T)\| + 1 + |y^{\lambda^+_b}(T)|)}\Bigg] + \frac{\eta}{\Delta}\\
    &\leq \kappa e^{\gamma(\|x\|+|y|)} + \frac{\eta}{\Delta},
    \end{align}
    for some $\hat{\xi}_1 \in (0, \Delta)$.
    Since $\kappa$ and $\gamma$ are independent of $\eta$, taking the limit as $\eta \rightarrow 0$ yields
    \begin{align}
    \frac{w^b(x+\Delta e_j, y, t) - w^b(x, y, t)}{\Delta} &\leq \kappa e^{\gamma(\|x\|+|y|)}.
    \end{align}
    Therefore, we conclude that
    $$
    |p^+_{j,\Delta}(x,y,t)| = \left|\frac{w^b(x+\Delta e_j, y, t) - w^b(x, y, t)}{\Delta}\right| \leq \kappa e^{\gamma(\|x\|+|y|)}.
    $$
           Similarly, we have
    \begin{align}
    |p^-_{j,\Delta}(x,y,t)| &= \left|\frac{w^b(x, y, t) - w^b(x-\Delta e_j, y, t)}{\Delta}\right| \leq \kappa e^{\gamma(\|x\|+|y|)},\\
    |q^+_\Delta(x,y,t)| &= \left|\frac{w^b(x, y+\Delta, t) - w^b(x, y, t)}{\Delta}\right| \leq \kappa e^{\gamma(\|x\|+|y|)},\\
    |q^-_\Delta(x,y,t)| &= \left|\frac{w^b(x, y, t) - w^b(x, y-\Delta, t)}{\Delta}\right| \leq \kappa e^{\gamma(\|x\|+|y|)}.
    \end{align}
    
    Additionally, we have
    \begin{align}
        s_{j,\Delta}(x,y,t) &= \frac{w^b(x+\Delta e_j,y,t) + w^b(x-\Delta e_j,y,t) - 2w^b(x,y,t)}{\Delta^2}\\
        &\geq \mathbb{E}\Bigg[\frac{f\big(x^{\lambda^\eta}(T)+\Delta e_j,y^{\lambda^\eta}(T)\big)+f\big(x^{\lambda^\eta}(T)-\Delta e_j,y^{\lambda^\eta}(T)\big)-2f\big(x^{\lambda^\eta}(T),y^{\lambda^\eta}(T)\big)}{\Delta^2}\Bigg] - \frac{2\eta}{\Delta^2}.
    \end{align}
    By the Mean Value Theorem, there exists $\xi_2 \in (-\Delta, \Delta)$ such that the term inside the expectation equals $f_{x_jx_j}(x^{\lambda^\eta}(T) + \xi_2 e_j, y^{\lambda^\eta}(T))$.

    By \eqref{growth_partial_derivatives_terminal} and the same argument for the first-order terms, we have
    $$
    |f_{x_jx_j}(x^{\lambda^\eta}(T) + \xi_2 e_j, y^{\lambda^\eta}(T))| \leq K e^{\gamma(\|x^{\lambda^\eta}(T)\| + 1 + |y^{\lambda^\eta}(T)|)},
    $$
    and it follows analogously that
    $$
    s_{j,\Delta}(x,y,t) \geq -\kappa e^{\gamma(\|x\|+|y|)} - \frac{2\eta}{\Delta^2}.
    $$
    Since $\kappa$ and $\gamma$ are independent of $\eta$, taking the limit as $\eta \rightarrow 0$ yields
    $$
    s_{j,\Delta}(x,y,t) \geq -\kappa e^{\gamma(\|x\|+|y|)}.
    $$
    Since $w^b$ is concave from Lemma \ref{lemma_sec_derivative_w^b}, we have
    \begin{align}
        s_{j,\Delta}(x,y,t) &= \frac{w^b(x+\Delta e_j,y,t) + w^b(x-\Delta e_j,y,t) - 2w^b(x,y,t)}{\Delta^2}\\
        &= w^b_{x_jx_j}(x + \hat{\xi}_2 e_j,y,t) \leq 0,
    \end{align}
    for some $\hat{\xi}_2 \in (-\Delta, \Delta)$.
    This establishes the exponential growth condition for $s_{j,\Delta}$. 
    
    A strictly analogous argument yields:
    \[
    0 \geq r_\Delta(x,y,t) \geq -\kappa e^{\gamma(\|x\| + |y|)}.
    \]

    \textbf{Step 2: We show that $p^{-}_{j,\Delta}(x,y,t)$, $q^{-}_\Delta(x,y,t)$ are sub-solutions, while $p^{+}_{j,\Delta}(x,y,t)$, $q^{+}_\Delta(x,y,t)$, $s_{j,\Delta}(x,y,t)$, $r_\Delta(x,y,t)$ are super-solutions of PDE \eqref{linear_pde_w^b}.}\\
    Recall the notation:
    \begin{align*}
        L(D^2w, Dw, \lambda) &=\frac{1+\epsilon}{2} \Tr[\sigma \sigma^\top w_{xx}] + [\sigma \sigma^\top \nabla c(\lambda)]^\top w_{xy} + \frac{1}{2} \|\sigma^\top \nabla c(\lambda)\|^2 w_{yy} + \lambda^\top w_x + c(\lambda) w_y =: L(w, \lambda).
    \end{align*}
    Then we have
\begin{align}
    w^b_t(x,y,t) + L(w^b(x,y,t), \lambda_b^*(x,y,t)) &= 0,\\
    w^b_t(x+\Delta e_j,y,t) + L(w^b(x+\Delta e_j,y,t), \lambda_b^*(x,y,t)) &\leq 0,\\
    w^b_t(x-\Delta e_j,y,t) + L(w^b(x-\Delta e_j,y,t), \lambda_b^*(x,y,t)) &\leq 0,
\end{align}
and it follows that 
\begin{align}
    (p^{+}_{j,\Delta})_t(x,y,t) + L(p^{+}_{j,\Delta}(x,y,t), \lambda_b^*(x,y,t)) &\leq 0,\\
    (p^{-}_{j,\Delta})_t(x,y,t) + L(p^{-}_{j,\Delta}(x,y,t), \lambda_b^*(x,y,t)) &\geq 0.
\end{align}
Thus, following the same argument, we have $p^{-}_{j,\Delta}(x,y,t)$, $q^{-}_\Delta(x,y,t)$ are sub-solutions, while $p^{+}_{j,\Delta}(x,y,t)$, $q^{+}_\Delta(x,y,t)$ are super-solutions of PDE \eqref{linear_pde_w^b} with the following terminal conditions
\begin{align*}
    p^{+}_{j,\Delta}(x,y,T) &= \frac{U_P(\ell(x + \Delta e_j) - U_A^{-1}(y)) - U_P(\ell(x) - U_A^{-1}(y))}{\Delta}\\
    &= U_P'(\ell(x + \delta_1 e_j) - U_A^{-1}(y)) \ell_{x_j}(x + \delta_1 e_j) > 0,\\
    p^{-}_{j,\Delta}(x,y,T) &= \frac{U_P(\ell(x) - U_A^{-1}(y)) - U_P(\ell(x - \Delta e_j) - U_A^{-1}(y))}{\Delta}\\
    &= U_P'(\ell(x - \delta_2 e_j) - U_A^{-1}(y)) \ell_{x_j}(x - \delta_2 e_j),\\  
    q^{+}_\Delta(x,y,T) &= \frac{U_P(\ell(x) - U_A^{-1}(y + \Delta)) - U_P(\ell(x) - U_A^{-1}(y))}{\Delta}\\
    &= -U_P'(\ell(x) - U_A^{-1}(\delta_3))(U_A^{-1})'(\delta_3),\\ 
    q^{-}_\Delta(x,y,T) &= \frac{U_P(\ell(x) - U_A^{-1}(y)) - U_P(\ell(x) - U_A^{-1}(y - \Delta))}{\Delta}\\
    &= -U_P'(\ell(x) - U_A^{-1}(\delta_4))(U_A^{-1})'(\delta_4) < 0,
\end{align*}
for some $\delta_1, \delta_2 \in (0, \Delta)$, $\delta_3 \in (y,y+\Delta)$, $\delta_4 \in (y-\Delta,y)$.
We also have
\begin{align}
    w^b_t(x,y,t) + L(w^b(x,y,t), \lambda_b^*(x,y,t)) &= 0,\\
    w^b_t(x+\Delta e_j,y,t) + L(w^b(x+\Delta e_j,y,t), \lambda_b^*(x,y,t)) &\leq 0,\\
    w^b_t(x-\Delta e_j,y,t) + L(w^b(x-\Delta e_j,y,t), \lambda_b^*(x,y,t)) &\leq 0.
\end{align}
It follows that 
$$
(s_{j,\Delta})_t(x,y,t) + L(s_{j,\Delta}(x,y,t), \lambda_b^*(x,y,t)) \leq 0,
$$
showing that $s_{j,\Delta}(x,y,t)$ is a supersolution of the PDE \eqref{linear_pde_w^b} with terminal condition 
\begin{align*}
    s_{j,\Delta}(x,y,T) &= \frac{U_P(\ell(x + \Delta e_j) - U_A^{-1}(y)) + U_P(\ell(x - \Delta e_j) - U_A^{-1}(y)) - 2 U_P(\ell(x) - U_A^{-1}(y))}{\Delta^2}\\     &= U_P''(\ell(x + \delta_5 e_j) - U_A^{-1}(y)) \big(\ell_{x_j}(x + \delta_5 e_j)\big)^2 + U_P'(\ell(x + \delta_5 e_j) - U_A^{-1}(y)) \ell_{x_j x_j}(x + \delta_5 e_j),\\
\end{align*}
for some $\delta_5 \in (-\Delta,\Delta)$. Similarly, $r_\Delta(x,y,t)$ is also a supersolution of the PDE \eqref{linear_pde_w^b} with terminal condition
\begin{align*}
    r_\Delta(x,y,T) &= \frac{U_P(\ell(x) - U_A^{-1}(y + \Delta)) + U_P(\ell(x) - U_A^{-1}(y - \Delta)) - 2 U_P(\ell(x) - U_A^{-1}(y))}{\Delta^2}\\     &= U_P''(\ell(x) - U_A^{-1}(\delta_6))\cdot \big((U_A^{-1})'(\delta_6)\big)^2 - U_P'(\ell(x) - U_A^{-1}(\delta_6))\cdot(U_A^{-1})''(\delta_6),
\end{align*}
for some $\delta_6 \in (y-\Delta,y+\Delta)$.
\end{proof}

\section{Proof of Theorem \ref{thm_cvg_w_i_w}}
\begin{proof}
\textbf{Proof of (i):} On the one hand, since $\mathcal{A}_k^b(x, y, t) \subset \mathcal{A}(x, y, t)$, we have $w^\epsilon(x,y,t) \geq w_k^{b}(x,y,t)$.

On the other hand, fix $\hat{\epsilon} > 0$ and $(x,y,t) \in \mathbb{R}^d \times \mathbb{R} \times [0,T]$, and choose $\Lambda = \{\lambda(s)\}_{s=t}^T \in  \mathcal{A}(x, y, t)$ satisfying
\begin{align*}
    w^\epsilon(x, y, t)&=\sup _{\Lambda\in \mathcal{A}(x, y, t)} \mathbb{E}\left[U_P\big(\ell \left(x^{\lambda}(T)\right)-U_A^{-1}(y^{\lambda}(T))\big)\right]\\
    &\leq \mathbb{E}\left[U_P\big(\ell \left(x^{\lambda}(T)\right)-U_A^{-1}(y^{\lambda}(T))\big)\right] + \hat{\epsilon}.
\end{align*}
For every $k \in \mathbb{N}$ and each component $i \in \{1,2,\dots,d\}$, define $\lambda_{k,i} = \lambda_i \mathbf{1}_{\{\frac{1}{k} \leq \lambda_i \leq k\}} + \frac{1}{k} \mathbf{1}_{\{\lambda_i < \frac{1}{k}\}} + k \mathbf{1}_{\{\lambda_i > k\}}$. By construction, $\Lambda_k = \{\lambda_k(s)\}_{s=t}^T \in \mathcal{A}^b_k(x, y, t)$.
We first establish the following lemma.
\begin{lemma} \label{lemma_cvg_lambda}
    \begin{align*}
        \lim_{k \rightarrow \infty}\mathbb{E}\left[\int_t^T \left\|\lambda(s) - \lambda_k(s)\right\|^2 ds\right] &= \lim_{k \rightarrow \infty}\mathbb{E}\left[\int_t^T \bigg(c\big(\lambda(s)\big) - c\big(\lambda_k(s)\big)\bigg)^2 ds\right]\\
        &= \lim_{k \rightarrow \infty}\mathbb{E}\left[\int_t^T \left\|\nabla c\big(\lambda(s)\big) - \nabla c\big(\lambda_k(s)\big)\right\|^2 ds\right] = 0.
    \end{align*}
\end{lemma}
\begin{proof} [Proof of Lemma \ref{lemma_cvg_lambda}]
    By \eqref{assumption_lambda} and Fubini's theorem, we have
    \begin{equation}
    \mathbb{E}\left[\left\|\lambda(s)\right\|^2\right] < \infty \quad \text{for almost every} \ s \in [t,T].
    \end{equation}
    
    For each component $i$, it follows from Fubini's theorem that
    \begin{align*}
        &\lim_{k \rightarrow \infty}\mathbb{E}\left[\int_t^T \left(\lambda_i(s) - \lambda_{k,i}(s)\right)^2 ds\right]\\
        =& \lim_{k \rightarrow \infty}\mathbb{E}\left[\int_t^T \left(\lambda_i(s) - k\right)^2  \mathbf{1}_{\{\lambda_i > k\}} ds + \int_t^T \left(\lambda_i(s) - \frac{1}{k}\right)^2  \mathbf{1}_{\{\lambda_i < \frac{1}{k}\}} ds\right]\\
        =& \lim_{k \rightarrow \infty}\int_t^T \mathbb{E}\left[\left(\lambda_i(s)-k\right)^2  \mathbf{1}_{\{\lambda_i > k\}}\right] ds + \lim_{k \rightarrow \infty} \int_t^T \mathbb{E}\left[\left(\lambda_i(s)- \frac{1}{k}\right)^2  \mathbf{1}_{\{\lambda_i < \frac{1}{k}\}}\right] ds.
    \end{align*}
    Notice that we have $\mathbb{E}\left[\left(\lambda_i(s)-k\right)^2  \mathbf{1}_{\{\lambda_i > k\}}\right] \leq \mathbb{E}\left[\left(\lambda_i(s)\right)^2\right]$ and $\mathbb{E}\left[\left(\lambda_i(s)-\frac{1}{k}\right)^2  \mathbf{1}_{\{\lambda_i < \frac{1}{k}\}}\right] \leq 1$. It follows from Fubini's theorem that $\int_t^T\mathbb{E}\left[\left(\lambda_i(s)\right)^2\right]ds = \mathbb{E}\left[\int_t^T\left(\lambda_i(s)\right)^2 ds\right] < \infty$. Hence, by the Dominated Convergence Theorem, we have
    \begin{align*}
    \lim_{k \rightarrow \infty}\int_t^T \mathbb{E}\left[\left(\lambda_i(s)-k\right)^2  \mathbf{1}_{\{\lambda_i > k\}}\right] ds &= \int_t^T \lim_{k \rightarrow \infty} \mathbb{E}\left[\left(\lambda_i(s)-k\right)^2  \mathbf{1}_{\{\lambda_i > k\}}\right] ds\\
    &= \int_t^T \lim_{k \rightarrow \infty} \mathbb{E}\left[\left(\lambda_i(s)-k\right)^2  \mathbf{1}_{\{\lambda_i > k\}}\right] \mathbf{1}_{\{\mathbb{E}\left[\left(\lambda_i(s)\right)^2\right] < \infty\}} ds,
    \end{align*}
    and
    \begin{align*}
    \lim_{k \rightarrow \infty}\int_t^T \mathbb{E}\left[\left(\lambda_i(s)-\frac{1}{k}\right)^2  \mathbf{1}_{\{\lambda_i < \frac{1}{k}\}}\right] ds &= \int_t^T \lim_{k \rightarrow \infty} \mathbb{E}\left[\left(\lambda_i(s)-\frac{1}{k}\right)^2  \mathbf{1}_{\{\lambda_i < \frac{1}{k}\}}\right] ds\\
    &= \int_t^T \lim_{k \rightarrow \infty} \mathbb{E}\left[\left(\lambda_i(s)-\frac{1}{k}\right)^2  \mathbf{1}_{\{\lambda_i < \frac{1}{k}\}}\right] \mathbf{1}_{\{\mathbb{E}\left[\left(\lambda_i(s)\right)^2\right] < \infty\}} ds.
    \end{align*}
    Notice that $\left(\lambda_i(s)-k\right)^2 \mathbf{1}_{\{\lambda_i > k\}} \leq \left(\lambda_i(s)\right)^2$ and $\left(\lambda_i(s)-\frac{1}{k}\right)^2  \mathbf{1}_{\{\lambda_i < \frac{1}{k}\}} \leq 1$. Moreover, we have $\mathbb{E}\left[\left(\lambda_i(s)\right)^2\right] < \infty$ for almost every $s \in [t,T]$     Then by the Dominated Convergence Theorem, we have $$
    \lim_{k \rightarrow \infty} \mathbb{E}\left[\left(\lambda_i(s)-k\right)^2  \mathbf{1}_{\{\lambda_i > k\}}\right] = \mathbb{E}\left[\lim_{k \rightarrow \infty}  \left(\lambda_i(s)-k\right)^2  \mathbf{1}_{\{\lambda_i > k\}}\right] = 0
    $$
    and 
    $$
    \lim_{k \rightarrow \infty} \mathbb{E}\left[\left(\lambda_i(s)-\frac{1}{k}\right)^2  \mathbf{1}_{\{\lambda_i < \frac{1}{k}\}}\right] = \mathbb{E}\left[\lim_{k \rightarrow \infty}  \left(\lambda_i(s)-\frac{1}{k}\right)^2  \mathbf{1}_{\{\lambda_i < \frac{1}{k}\}}\right] = 0.
    $$
    Thus, $\lim_{k \rightarrow \infty}\mathbb{E}\left[\int_t^T \left(\lambda_i(s) - \lambda_{k,i}(s)\right)^2 ds\right] = 0$. Summing over all components $i$, we obtain 
    $$
    \lim_{k \rightarrow \infty}\mathbb{E}\left[\int_t^T \left\|\lambda(s) - \lambda_k(s)\right\|^2 ds\right] = 0.
    $$

    Additionally, it follows from the Cauchy--Schwarz inequality that 
    $$\bigg(c\big(\lambda(s)\big) - c\big(\lambda_k(s)\big)\bigg)^2 = \bigg(\sum_{i=1}^d \left(c_i\big(\lambda_i(s)\big) - c_i\big(\lambda_{k,i}(s)\big)\right)\bigg)^2 \leq d \sum_{i=1}^d \left(c_i\big(\lambda_i(s)\big) - c_i\big(\lambda_{k,i}(s)\big)\right)^2.
    $$ Thus,
    $\lim_{k \rightarrow \infty}\mathbb{E}\left[\int_t^T \bigg(c\big(\lambda(s)\big) - c\big(\lambda_k(s)\big)\bigg)^2 ds\right] = \lim_{k \rightarrow \infty}\mathbb{E}\left[\int_t^T \left\|\nabla c\big(\lambda(s)\big) - \nabla c\big(\lambda_k(s)\big)\right\|^2 ds\right] = 0$ follow from the same argument.
\end{proof}

\begin{lemma} \label{lemma_cvg_xy}
    $\lim_{k \rightarrow \infty}\mathbb{E}\left[\left|\ell \left(x^{\lambda_k}(T)\right)-\ell \left(x^{\lambda}(T)\right)\right|\right] = 0$ and $\lim_{k \rightarrow \infty}\mathbb{E}\left[\left|y^{\lambda_k}(T)-y^{\lambda}(T)\right|\right] = 0$.
\end{lemma}
\begin{proof}[Proof of Lemma \ref{lemma_cvg_xy}]
We first show that 
$$
\lim_{k \rightarrow \infty}\mathbb{E}\left[\left\|x^{\lambda_k}(T)-x^{\lambda}(T)\right\|^2\right] = \lim_{k \rightarrow \infty}\mathbb{E}\left[\left(y^{\lambda_k}(T)-y^{\lambda}(T)\right)^2\right] = 0.
$$
Notice that we have
\begin{align*}
    x^{\lambda_k}(T) &= x + \int_t^T\lambda_k(s)ds + \int_t^T \sigma d\mathcal{B}(s) + \int_t^T \sqrt{\epsilon} \sigma d\mathcal{W}(s),\\
    y^{\lambda_k}(T) &= y + \int_t^T c\left(\lambda_k(s)\right)ds + \int_t^T \nabla c\left(\lambda_k(s)\right)^\top \sigma d\mathcal{B}(s),\\
    x^{\lambda}(T) &= x + \int_t^T\lambda(s)ds + \int_t^T \sigma d\mathcal{B}(s) + \int_t^T \sqrt{\epsilon} \sigma d\mathcal{W}(s),\\
    y^{\lambda}(T) &= y + \int_t^T c\left(\lambda(s)\right)ds + \int_t^T \nabla c\left(\lambda(s)\right)^\top \sigma d\mathcal{B}(s).
\end{align*}
The Cauchy--Schwarz inequality yields
\begin{align*}
    \mathbb{E}\left[\left\|x^{\lambda_k}(T)-x^{\lambda}(T)\right\|^2\right] &= \mathbb{E}\left[\left\|\int_t^T \left(\lambda_k(s) - \lambda(s)\right) ds\right\|^2\right] \leq (T-t) \mathbb{E}\left[\int_t^T \left\|\lambda_k(s) - \lambda(s) \right\|^2 ds\right].
\end{align*}
Similarly,
\begin{align*}
    &\mathbb{E}\left[\left(y^{\lambda_k}(T)-y^{\lambda}(T)\right)^2\right]\\
    =& \mathbb{E}\left[\left( \int_t^T \left(c\left(\lambda_k(s)\right) - c\left(\lambda(s)\right)\right) ds + \int_t^T \left(\nabla c\left(\lambda_k(s)\right) - \nabla c\left(\lambda(s)\right)\right)^\top \sigma d\mathcal{B}(s)\right)^2\right]\\
    \leq& 2 \mathbb{E}\left[ \left(\int_t^T \left(c\left(\lambda_k(s)\right) - c\left(\lambda(s)\right)\right)ds\right)^2\right] + 2 \mathbb{E}\left[ \left(\int_t^T \left(\nabla c\left(\lambda_k(s)\right) - \nabla c\left(\lambda(s)\right)\right)^\top \sigma d\mathcal{B}(s)\right)^2\right]\\
    \leq& 2(T-t) \mathbb{E}\left[ \int_t^T \left(c\left(\lambda_k(s)\right) - c\left(\lambda(s)\right)\right)^2ds\right] + 2\mathbb{E}\left[ \int_t^T\left\|\sigma^\top \left(\nabla c\left(\lambda_k(s)\right) - \nabla c\left(\lambda(s)\right)\right)\right\|^2ds\right].
\end{align*}
Thus, by Lemma \ref{lemma_cvg_lambda}, we have $\lim_{k \rightarrow \infty}\mathbb{E}\left[\left\|x^{\lambda_k}(T)-x^{\lambda}(T)\right\|^2\right] = \lim_{k \rightarrow \infty}\mathbb{E}\left[\left(y^{\lambda_k}(T)-y^{\lambda}(T)\right)^2\right] = 0$. Moreover,
$\lim_{k \rightarrow \infty}\mathbb{E}\left[\left|\ell(x^{\lambda_k}(T))-\ell(x^{\lambda}(T))\right|\right] = 0$ follows from
$$
\mathbb{E}\left[\left|\ell(x^{\lambda_k}(T))-\ell(x^{\lambda}(T))\right|\right] \leq \sqrt{d} c_\ell \mathbb{E}\left[\left\|x^{\lambda_k}(T)-x^{\lambda}(T)\right\|\right] \leq \sqrt{d} c_\ell \sqrt{\mathbb{E}\left[\left\|x^{\lambda_k}(T)-x^{\lambda}(T)\right\|^2\right]},
$$
and $\lim_{k \rightarrow \infty}\mathbb{E}\left[\left|y^{\lambda_k}(T)-y^{\lambda}(T)\right|\right] = 0$ follows by the same argument.
\end{proof}

To prove (i), we modify the principal's utility $U_P$. To be more precise, we fix a constant $M >0$ and modify $U_P$ as follows:
\begin{equation} \label{new_UP}
    U_P^M(x) := \begin{cases}
        U_P(x), \quad x \in [-M,\infty),\\
        A(x), \quad x \in (-\infty,-M), \end{cases}
\end{equation}
where $A$ is an asymptotically linear function chosen so that $U_P^M \in C^{2+\alpha^0}(\mathbb{R})$, $U_P^M\geq U_P$ and Assumption \ref{assumptionUP} holds with constants independent of $M$. Therefore, $U_P^M$ still satisfies Assumption \ref{assumptionUP} and $(U_P^M)' \leq c^1_P$ for some constant $c^1_P > 0$. We also define the corresponding objective function and choose $\Lambda^M = \{\lambda^M(s)\}_{s=t}^T \in  \mathcal{A}(x, y, t)$ satisfying
\begin{align*}
    w^{\epsilon,M}(x, y, t)&=\sup _{\Lambda\in \mathcal{A}(x, y, t)} \mathbb{E}\left[U_P^M\big(\ell\left(x^{\lambda}(T)\right)-U_A^{-1}(y^{\lambda}(T))\big)\right]\\
    &\leq \mathbb{E}\left[U_P^M\big(\ell\left(x^{\lambda^{M}}(T)\right)-U_A^{-1}(y^{\lambda^{M}}(T))\big)\right] + \hat{\epsilon},\\
    w_k^{b,M}(x, y, t)&=\sup _{\Lambda\in \mathcal{A}_k^b(x, y, t)} \mathbb{E}\left[U_P^M\big(\ell\left(x^{\lambda}(T)\right)-U_A^{-1}(y^{\lambda}(T))\big)\right]\\
    &\geq \mathbb{E}\left[U_P^M\big(\ell\left(x^{\lambda_k^M}(T)\right)-U_A^{-1}(y^{\lambda_k^M}(T))\big)\right],
\end{align*}
where $\lambda_{k,i}^M = \lambda^{M}_i \mathbf{1}_{\{\frac{1}{k} \leq |\lambda^{M}_i|\leq k\}} + \frac{1}{k} \mathbf{1}_{\{|\lambda^{M}_i| < \frac{1}{k}\}} + k \mathbf{1}_{\{|\lambda^{M}_i|> k\}}$. Based on previous analysis, we have the following conclusions:
\begin{align}
    w_k^{b,M}(x,y,t) &\in C^{2+\alpha,\frac{2+\alpha}{2}}(\mathbb{R}^d \times \mathbb{R} \times [0,T]).\\
    \lim_{k \rightarrow \infty}\mathbb{E}\left[\int_t^T \left\|\lambda^{M}(s) - \lambda_k^M(s)\right\|^2 ds\right] &= \lim_{k \rightarrow \infty}\mathbb{E}\left[\int_t^T \bigg(c\big(\lambda^{M}(s)\big) - c\big(\lambda_k^M(s)\big)\bigg)^2 ds\right]\\
    &= \lim_{k \rightarrow \infty}\mathbb{E}\left[\int_t^T \left\|\nabla c\big(\lambda^{M}(s)\big) - \nabla c\big(\lambda_k^M(s)\big)\right\|^2 ds\right] = 0.\\
    \lim_{k \rightarrow \infty}\mathbb{E}\left[\left|\ell \left(x^{\lambda_k^M}(T)\right)-\ell \left(x^{\lambda^{M}}(T)\right)\right|\right] &= 0.\\
    \lim_{k \rightarrow \infty}\mathbb{E}\left[\left|y^{\lambda_k^M}(T)-y^{\lambda^{M}}(T)\right|\right] &= 0.\\
    \end{align}
We first state and prove the following two lemmas, which are crucial for proving (i).
\begin{lemma} \label{lemma_cvg_w^M}
    $\lim_{k \rightarrow \infty}w_k^{b,M}(x,y,t) = w^{\epsilon,M}(x,y,t), \quad \forall (x,y,t) \in \mathbb{R}^d \times \mathbb{R} \times [0,T]$.
\end{lemma}
\begin{proof} [Proof of Lemma \ref{lemma_cvg_w^M}]
On the one hand, since $\mathcal{A}_k^b(x, y, t) \subset \mathcal{A}(x, y, t)$, we have 
$$
w^{\epsilon,M}(x,y,t) \geq w_k^{b,M}(x,y,t).
$$
On the other hand, we have 
\begin{align*}
    & \quad w^{\epsilon,M}(x,y,t) - w_k^{b,M}(x,y,t)\\
    &= \sup _{\Lambda\in \mathcal{A}(x, y, t)} \mathbb{E}\left[U_P^M\big(\ell \left(x^{\lambda}(T)\right)-U_A^{-1}(y^{\lambda}(T))\big)\right] - \sup _{\Lambda\in \mathcal{A}_k^b(x, y, t)} \mathbb{E}\left[U_P^M\big(\ell \left(x^{\lambda}(T)\right)-U_A^{-1}(y^{\lambda}(T))\big)\right]\\
    &\leq \mathbb{E}\left[U_P^M\big(\ell \left(x^{\lambda^{M}}(T)\right)-U_A^{-1}(y^{\lambda^{M}}(T))\big)\right] - \mathbb{E}\left[U_P^M\big(\ell \left(x^{\lambda^M_k}(T)\right)-U_A^{-1}(y^{\lambda^M_k}(T))\big)\right] + \hat{\epsilon}\\
    & \leq c_P^1 \mathbb{E}\left[\left|\ell \left(x^{\lambda^{M}}(T)\right)- \ell \left(x^{\lambda^M_k}(T)\right) + U_A^{-1}(y^{\lambda^M_k}(T)) - U_A^{-1}(y^{\lambda^{M}}(T))\right|\right] + \hat{\epsilon}\\
    &\leq c_P^1 \mathbb{E}\left[\left|\ell \left(x^{\lambda^{M}}(T)\right)- \ell \left(x^{\lambda^M_k}(T)\right)\right|\right] + c_P^1 c \mathbb{E}\left[\left|y^{\lambda^{M}}(T)- y^{\lambda^M_k}(T)\right|\right] + \hat{\epsilon},
\end{align*}
where the last two inequalities follow from the boundedness of $(U^M_P)'$ and $(U_A^{-1})'$. Hence, the convergence is proved by letting $k \rightarrow \infty$ and $\hat{\epsilon} \rightarrow 0$.
\end{proof}
\begin{lemma} \label{lemma_w^bM_to_w^b}
    $\lim_{M \rightarrow \infty}w_{\tilde{c}_\lambda}^{b,M}(x,y,t) = w_{\tilde{c}_\lambda}^b(x,y,t)$.
\end{lemma}
\begin{proof}[Proof of Lemma \ref{lemma_w^bM_to_w^b}]
    Choose $\Lambda^{b,M} = \{\lambda^{b,M}(s)\}_{s=t}^T \in \mathcal{A}^b_{\tilde{c}_\lambda}(x, y, t)$ satisfying
    \begin{align*}
        w_{\tilde{c}_\lambda}^{b,M}(x, y, t)&=\sup _{\Lambda\in \mathcal{A}^b_{\tilde{c}_\lambda}(x, y, t)} \mathbb{E}\left[U_P^M\big(\ell \left(x^{\lambda}(T)\right)-U_A^{-1}(y^{\lambda}(T))\big)\right]\\
        &\leq \mathbb{E}\left[U_P^M\big(\ell \left(x^{\lambda^{b,M}}(T)\right)-U_A^{-1}(y^{\lambda^{b,M}}(T))\big)\right] + \hat{\epsilon}.
    \end{align*}
    Then we have
    \begin{align*}
        w_{\tilde{c}_\lambda}^{b}(x, y, t)&=\sup _{\Lambda\in \mathcal{A}_{\tilde{c}_\lambda}^b(x, y, t)} \mathbb{E}\left[U_P\big(\ell \left(x^{\lambda}(T)\right)-U_A^{-1}(y^{\lambda}(T))\big)\right]\\
        &\geq \mathbb{E}\left[U_P\big(\ell \left(x^{\lambda^{b,M}}(T)\right)-U_A^{-1}(y^{\lambda^{b,M}}(T))\big)\right].
    \end{align*}
    Let $z(T) := \ell \left(x^{\lambda^{b,M}}(T)\right)-U_A^{-1}(y^{\lambda^{b,M}}(T))$. Then
    \begin{align}
        0 &\leq w_{{\tilde{c}_\lambda}}^{b,M}(x,y,t) - w_{{\tilde{c}_\lambda}}^b(x,y,t)\\
        &= \sup _{\Lambda\in \mathcal{A}_{\tilde{c}_\lambda}^b(x, y, t)} \mathbb{E}\left[U_P^M\big(\ell \left(x^{\lambda}(T)\right)-U_A^{-1}(y^{\lambda}(T))\big)\right] - \sup _{\Lambda\in \mathcal{A}_{\tilde{c}_\lambda}^b(x, y, t)} \mathbb{E}\left[U_P\big(\ell \left(x^{\lambda}(T)\right)-U_A^{-1}(y^{\lambda}(T))\big)\right]\\
        &\leq \mathbb{E}\left[U_P^M\big(\ell \left(x^{\lambda^{b,M}}(T)\right)-U_A^{-1}(y^{\lambda^{b,M}}(T))\big)\right] - \mathbb{E}\left[U_P\big(\ell \left(x^{\lambda^{b,M}}(T)\right)-U_A^{-1}(y^{\lambda^{b,M}}(T))\big)\right] + \hat{\epsilon}\\
        &= \mathbb{E}\left[\big(U_P^M(z(T)) - U_P(z(T))\big)\mathbf{1}_{\{z(T)<-M\}}\right] + \hat{\epsilon}\\
        &\leq C_P\mathbb{E}\left[e^{-c_P z(T)}\mathbf{1}_{\{z(T)<-M\}}\right] + \hat{\epsilon}, \label{tail_prob}
    \end{align}
    for some constant $C_P > 0$.
    Notice that by Itô's formula, we have
    \begin{align}
    &d \left(U_A^{-1}(y^{\lambda^{b,M}}(s))\right)\\
    =& \bigg((U_A^{-1})'c(\lambda^{b,M}(s)) + \frac{1}{2} (U_A^{-1})''\left\|\sigma^\top \nabla c(\lambda^{b,M}(s))\right\|^2\bigg) d s+\left((U_A^{-1})' \nabla c(\lambda^{b,M}(s))^\top \sigma \right) d \mathcal{B}(s).
    \end{align}
    Thus, it follows that
    \begin{align}
    dz(s) &= \bigg(\nabla \ell^\top \lambda^{b,M}(s) + \frac{1+\epsilon}{2} \Tr[\sigma \sigma^\top \nabla^2 \ell] - (U_A^{-1})'c(\lambda^{b,M}(s)) - \frac{1}{2} (U_A^{-1})''\big\|\sigma^\top \nabla c(\lambda^{b,M}(s))\big\|^2\bigg) d s\\
    &+\left(\nabla \ell^\top \sigma - (U_A^{-1})' \nabla c(\lambda^{b,M}(s))^\top \sigma\right) d \mathcal{B}(s) + \sqrt{\epsilon} \nabla \ell^\top \sigma d\mathcal{W}(s),
    \end{align}
    where all the coefficients are uniformly bounded, independent of $M$. Therefore, we have
    \begin{align} \label{dym_z}
        \mathbb{P}\left(z(T) < -V\right) &\leq \frac{1}{\delta} e^{-\delta V^2}, \forall V > 0,
    \end{align}
    where $\delta > 0$ is independent of $V$. 
    Hence, setting $u = \frac{\ln a}{c_P}$, we obtain
    \begin{align}
        \mathbb{E}\left[e^{-c_P z(T)}\mathbf{1}_{z(T)<-M}\right] &= \int_0^\infty \mathbb{P}(e^{-c_P z(T)}\mathbf{1}_{\{z(T)<-M\}} > a)da\\
        &= \int_0^{e^{c_P M}} \mathbb{P}(z(T) < -M) \, da + \int_{e^{c_P M}}^\infty \mathbb{P}\left(z(T) < -\frac{\ln a}{c_P}\right) da\\
        &\leq \frac{1}{\delta} e^{-\delta M^2 + c_P M} + c_P\int_{M}^\infty \mathbb{P}\left(z(T) < -u\right)e^{c_Pu} du\\
        &\leq \frac{1}{\delta} e^{-\delta M^2 + c_P M} + \frac{c_P}{\delta}\int_{M}^\infty e^{-\delta u^2 + c_Pu}du\\
        &= \frac{1}{\delta} e^{-\delta M^2 + c_P M} + \frac{c_P}{\delta} e^{\frac{c_P^2}{4\delta}}\int_{M}^\infty e^{- \delta(u-\frac{c_P}{2\delta})^2}du\\
        &= \frac{1}{\delta} e^{-\delta M^2 + c_P M} + \frac{c_P}{\delta} e^{\frac{c_P^2}{4\delta}} \sqrt{\frac{\pi}{\delta}}\int_{M}^\infty \sqrt{\frac{\delta}{\pi}}e^{- \delta(u-\frac{c_P}{2\delta})^2}du\\
        &=\frac{1}{\delta} e^{-\delta M^2 + c_P M} + \frac{c_P}{\delta} e^{\frac{c_P^2}{4\delta}} \sqrt{\frac{\pi}{\delta}}\left[1-\Phi\left(\left(M-\frac{c_P}{2\delta}\right)\sqrt{2\delta}\right)\right]\\
        &=:R(M).
    \end{align}
    Since both terms in $R(M)$ converge to $0$ as $M\rightarrow\infty$, letting $M\rightarrow\infty$ and then $\hat{\epsilon}\rightarrow 0$ in \eqref{tail_prob} completes the proof. 
    \end{proof}

On the one hand, it follows from Proposition \ref{prop_unique_lambda_w^b} and Lemma \ref{lemma_cvg_w^M} that
\begin{align}
    \lim_{k \rightarrow \infty}w_k^b(x,y,t) &= w_{\tilde{c}_\lambda}^b(x,y,t),\\
    w^{\epsilon,M}(x,y,t) = \lim_{k \rightarrow \infty}w_k^{b,M}(x,y,t) &= w_{\tilde{c}_\lambda}^{b,M}(x,y,t), \quad \forall M \geq 0.
\end{align}
On the other hand, Lemma \ref{lemma_w^bM_to_w^b} guarantees that 
$$
\lim_{M \rightarrow \infty}w_{\tilde{c}_\lambda}^{b,M}(x,y,t) = w_{\tilde{c}_\lambda}^b(x,y,t).
$$ 
Additionally, since $U_P^M\geq U_P$ by construction, we have
$w^{\epsilon,M}\geq w^\epsilon$ for every $M$. Together with Lemma \ref{lemma_w^bM_to_w^b}, the following limit exists and satisfies
\begin{equation} \label{w^epsilon,M_geq_w^epsilon}
    \lim_{M\to\infty} w^{\epsilon,M} \geq w^\epsilon.
\end{equation}
Therefore, 
$$
 \lim_{k \rightarrow \infty}w_k^b(x,y,t) \geq w^\epsilon
$$
and the proof of (i) is completed.
\begin{figure}[H]
\centering
\begin{tikzpicture}[scale=0.7, transform shape,
    block/.style={rectangle, draw=black, thick, fill=white,
                  text width=2.5cm, minimum height=1cm, align=center,
                  inner sep=8pt},
    arrow/.style={thick, -Stealth[scale=1.2]},
    doublearrow/.style={thick, <->, -Stealth[scale=1.2]},
    label/.style={midway, fill=white, inner sep=2pt, font=\small}
]
\node[block, text width=2cm] (A) at (0,0) {$\lim_{k \rightarrow \infty} w^b_k$};
\node[block, text width=1cm, right=2.5cm of A] (B) {$w^b_{\tilde{c}_\lambda}$};
\node[block, right=2.5cm of B] (C) {$\lim_{M \rightarrow \infty} w^{b,M}_{\tilde{c}_\lambda}$};
\node[block, right=2.5cm of C] (D) {$\lim_{M \rightarrow \infty} w^{\epsilon,M}$};
\node[block, text width=1cm, right=3cm of D] (E) {$w^\epsilon$};
\draw[thick, <->] (A) -- node[label, above] {Proposition \ref{prop_unique_lambda_w^b}} (B);
\draw[thick, <->] (B) -- node[label, above] {Lemma \ref{lemma_w^bM_to_w^b}} (C);
\draw[thick, <->] (C) -- node[label, above] {Lemma \ref{lemma_cvg_w^M}} (D);
\draw[thick, ->] (D) -- node[label, above] {$\geq$ (see \eqref{w^epsilon,M_geq_w^epsilon})} (E);
\end{tikzpicture}
\caption{Logic flow}
\label{logic_1}
\end{figure}

\textbf{Proof of (ii):} 
$\forall k \geq \tilde{c}_\lambda$, $w^b_k$ is the classical solution of the PDE \eqref{pde_w_bdd_lamb} with $k = \tilde{c}_\lambda$ since the F.O.C. \eqref{FOC_w^b} yields the optimal effort when $k \geq \tilde{c}_\lambda$. By the comparison principle for PDE \eqref{pde_w_bdd_lamb} with $k = \tilde{c}_\lambda$, we have $w^b_k = w^b_{\tilde{c}_\lambda}, \ \forall k \geq \tilde{c}_\lambda$. Moreover, since $\lim_{k \rightarrow \infty} w_k^b = w^\epsilon$ from (i), it follows that $w^\epsilon = w^b_{\tilde{c}_\lambda}$. Hence, $w^\epsilon \in C^{2+\alpha,\frac{2+\alpha}{2}}\bigg(\mathbb{R}^d \times \mathbb{R} \times [0,T]\bigg)$, which, together with Proposition \ref{prop_unique_lambda_w^b}, proves (ii).

\textbf{Proof of (iii):} This is a direct consequence of (ii) and Proposition \ref{prop_unique_lambda_w^b}(2).
\end{proof}

\section{Proof of Theorem \ref{thm_regularity_w^0}}
\begin{proof}
    Fix $ (x_0,y_0) \in \mathbb{R}^d \times \mathbb{R}$ and $ R > 0$, and consider the cylinder $Q = B_R(x_0,y_0) \times (0,T)$ where $B_R(x_0,y_0) := \{(x,y)|\|x-x_0\|^2+(y-y_0)^2 < R^2\}$.
    We proceed in two steps. First, we show that, up to a subsequence, $w^\epsilon \xrightarrow{w} w^0$ in $W^{2,1}_p(\overline{Q})$ as $\epsilon \rightarrow 0$. Moreover, $\|w^0\|_{W^{2,1}_p(\overline{Q})} \leq \tilde{C}_Q$ for some $\tilde{C}_Q > 0$ that is independent of $\epsilon$. Second, we demonstrate that $w^0 \in W^{2,1}_{\infty}(\overline{Q})$.

    \textbf{Step 1:} It follows directly from Theorem \ref{thm_cvg_w_i_w}(iii) that $\forall 1 < p < \infty$,
    $\|w^\epsilon\|_{W^{2,1}_p(\overline{Q})} \leq \tilde{C}_Q$ for some $\tilde{C}_Q > 0$ that is independent of $\epsilon$. Define the bounded, closed and convex set $B_C := \{w \in W^{2,1}_p(\overline{Q}) \mid \|w\|_{W^{2,1}_p(\overline{Q})} \leq \tilde{C}_Q\}$. By the reflexivity of $W_p^{2,1}(\overline{Q})$, the set $B_C$ is weakly sequentially compact. Hence, up to a subsequence, $w^\epsilon \xrightarrow{w} v \in  B_C$ in $W^{2,1}_p(\overline{Q})$, as $\epsilon \rightarrow 0$.

                                It remains to show that $\lim_{\epsilon \rightarrow 0}w^\epsilon(x,y,t) = w^0(x,y,t)$ for every $(x,y,t) \in \overline{Q}$. 
            
    To this end, on the one hand, Theorem \ref{thm_cvg_w_i_w} shows that $w^\epsilon(x,y,t) = w^b_{\tilde{c}_\lambda}(x,y,t)$ is a classical solution of the PDE \eqref{pde_w_bdd_lamb}, and hence
    \begin{align}
        w^\epsilon_t + \sup_{\lambda \in [\frac{1}{\tilde{c}_\lambda},\tilde{c}_\lambda]^d}
        \Bigg\{
        \frac{1}{2} \Tr[\sigma \sigma^\top w^\epsilon_{xx}] + [\sigma \sigma^\top \nabla c(\lambda)]^\top w^\epsilon_{xy} + \frac{1}{2} \|\sigma^\top \nabla c(\lambda)\|^2 w^\epsilon_{yy} +& \lambda^\top w^\epsilon_x + c(\lambda) w^\epsilon_y \Bigg\}\\
        &= -\frac{\epsilon}{2} \Tr[\sigma \sigma^\top w^\epsilon_{xx}] \geq 0,
    \end{align}
    since $w^\epsilon(x,y,t)$ is concave on $\mathbb{R}^d \times \mathbb{R}$ for any $t \in [0,T]$. Thus, $w^\epsilon(x,y,t)$ is a viscosity subsolution of the above PDE. Moreover, since $w^0(x,y,t)$ is the viscosity solution of PDE \eqref{pde_w_noepsilon}, it follows that $w^0(x,y,t)$ is a viscosity supersolution of the above PDE, i.e., 
    \begin{align}
         w^0_t + \sup_{\lambda \in [\frac{1}{\tilde{c}_\lambda},\tilde{c}_\lambda]^d}
        \Bigg\{
        \frac{1}{2} \Tr[\sigma \sigma^\top w^0_{xx}] + [\sigma \sigma^\top \nabla c(\lambda)]^\top w^0_{xy} + \frac{1}{2} \|\sigma^\top \nabla c(\lambda)\|^2 w^0_{yy} + \lambda^\top w^0_x + c(\lambda) w^0_y \Bigg\} \leq 0.
    \end{align}
    It follows from the comparison principle for the counterpart of PDE \eqref{pde_w_bdd_lamb} with $\epsilon=0$ that 
    \begin{equation} \label{w^epsilon_leq_w^0}
        w^\epsilon(x,y,t) \leq w^0(x,y,t).
    \end{equation}
    
    On the other hand, let $w^{0,M}(x,y,t)$ denote the counterpart of $w^0(x,y,t)$ under $U_P^M$. Take $\lambda^{0,M} \in \mathcal{A}^0(x, y, t) \subset \mathcal{A}(x, y, t)$ satisfying 
        \begin{align}
    w^{0,M}(x,y,t) &= \sup _{\Lambda\in \mathcal{A}^0(x, y, t)} \mathbb{E}\left[U_P^M\big(\ell\left(\hat{x}^{\lambda}(T)\right)-U_A^{-1}(y^{\lambda}(T))\big)\right]\\
    &\leq \mathbb{E}\left[U_P^M\big(\ell\left(\hat{x}^{\lambda^{0,M}}(T)\right)-U_A^{-1}(y^{\lambda^{0,M}}(T))\big)\right] + \hat{\epsilon}.
    \end{align}
    Using the same argument as in the proof of \eqref{w^epsilon_leq_w^0}, we have
    \begin{equation} \label{w^epsilon,M_leq_w^0,M}
        w^{\epsilon,M}(x,y,t) \leq w^{0,M}(x,y,t).
    \end{equation}
    Additionally, we have the following two lemmas.
    \begin{lemma} \label{lemma_cvg_w^M0}
         $\lim_{\epsilon \rightarrow 0} w^{\epsilon,M}(x,y,t) = w^{0,M}(x,y,t).$
    \end{lemma}
    \begin{proof}[Proof of Lemma \ref{lemma_cvg_w^M0}]
        \begin{align} \label{cvg_hat_x_to_x}
        \mathbb{E}\left[\left|\ell\left(\hat{x}^{\lambda^{0,M}}(T)\right) - \ell\left(x^{\lambda^{0,M}}(T)\right)\right|\right] 
        &\leq c_\ell \sum_{j=1}^d \mathbb{E}\left[\left|\hat{x}^{\lambda^{0,M}}_j(T) - x^{\lambda^{0,M}}_j(T)\right|\right]\\
        &= \sqrt{\epsilon} c_\ell \sum_{j=1}^d \mathbb{E}\left[\left|\sigma_{j} \left(\mathcal{W}_j(T)-\mathcal{W}_j(t)\right)\right|\right]\\
        &= \sqrt{\epsilon} c_\ell \sum_{j=1}^d |\sigma_j| \sqrt{(T-t)} \sqrt{\frac{2}{\pi}},
        \end{align}
        where the last equality follows from the expectation of the folded normal distribution since 
        $$
        \sigma_j\left(\mathcal{W}_j(T)-\mathcal{W}_j(t)\right) \sim \mathcal{N}\left(0, \sigma_j^2(T-t)\right).
        $$
        Therefore, it follows immediately that $\lim_{\epsilon \rightarrow 0}\mathbb{E}\left[\left|\ell\left(\hat{x}^{\lambda^{0,M}}(T)\right) - \ell\left(x^{\lambda^{0,M}}(T)\right)\right|\right] = 0$.
        It follows from \eqref{w^epsilon,M_leq_w^0,M} that
        \begin{align}
        0 &\leq w^{0,M}(x,y,t) - w^{\epsilon,M}(x,y,t)\\
        & = \sup _{\Lambda\in \mathcal{A}^0(x, y, t)} \mathbb{E}\left[U_P^M\big(\ell\left(\hat{x}^{\lambda}(T)\right)-U_A^{-1}(y^{\lambda}(T))\big)\right] - \sup _{\Lambda\in \mathcal{A}(x, y, t)} \mathbb{E}\left[U_P^M\big(\ell\left(x^{\lambda}(T)\right)-U_A^{-1}(y^{\lambda}(T))\big)\right]\\
        &\leq \mathbb{E}\left[U_P^M\big(\ell\left(\hat{x}^{\lambda^{0,M}}(T)\right)-U_A^{-1}(y^{\lambda^{0,M}}(T))\big)\right] - \mathbb{E}\left[U_P^M\big(\ell\left(x^{\lambda^{0,M}}(T)\right)-U_A^{-1}(y^{\lambda^{0,M}}(T))\big)\right] + \hat{\epsilon}\\
        &\leq c_P^1 \mathbb{E}\left[\left|\ell\left(\hat{x}^{\lambda^{0,M}}(T)\right) - \ell\left(x^{\lambda^{0,M}}(T)\right)\right|\right] + \hat{\epsilon}.
    \end{align}
    Thus, $0 \leq \liminf_{\epsilon \rightarrow 0} \big(w^{0,M}(x,y,t) - w^{\epsilon,M}(x,y,t)\big) \leq \limsup_{\epsilon \rightarrow 0} \big(w^{0,M}(x,y,t) - w^{\epsilon,M}(x,y,t)\big) \leq \hat{\epsilon}$ by \eqref{cvg_hat_x_to_x}. Since $\hat{\epsilon} > 0$ is chosen arbitrarily, letting $\hat{\epsilon} \rightarrow 0$ implies $\lim_{\epsilon \rightarrow 0} w^{\epsilon,M} = w^{0,M}$.
    \end{proof}
\begin{lemma}\label{lemma_cvg_wM_to_w}
    $$
    \lim_{M\rightarrow \infty}w^{\epsilon,M}(x,y,t) = w^\epsilon(x,y,t), \quad \text{uniformly in} \,\, \epsilon.
    $$
\end{lemma}
\begin{proof}[Proof of Lemma \ref{lemma_cvg_wM_to_w}]
    Indeed, by the construction of $U_P^M$, $w^\epsilon(x,y,t) \leq w^{\epsilon,M}(x,y,t) = w^{b,M}_{\tilde{c}_\lambda}(x,y,t)$. Moreover, by Theorem \ref{thm_cvg_w_i_w}, we have $w^\epsilon(x,y,t) = w^b_{\tilde{c}_\lambda}(x,y,t)$. Thus, denoting $z(T) := \ell \left(x^{\lambda^{b,M}}(T)\right)-U_A^{-1}(y^{\lambda^{b,M}}(T))$, we have
    \begin{align}
        0 &\leq w^{\epsilon,M}(x,y,t) - w^\epsilon(x,y,t) = w_{\tilde{c}_\lambda}^{b,M}(x,y,t) - w_{\tilde{c}_\lambda}^b(x,y,t)\\
        &= \sup _{\Lambda\in \mathcal{A}_{\tilde{c}_\lambda}^b(x, y, t)} \mathbb{E}\left[U_P^M\big(\ell\left(x^{\lambda}(T)\right)-U_A^{-1}(y^{\lambda}(T))\big)\right] - \sup _{\Lambda\in \mathcal{A}_{\tilde{c}_\lambda}^b(x, y, t)} \mathbb{E}\left[U_P\big(\ell\left(x^{\lambda}(T)\right)-U_A^{-1}(y^{\lambda}(T))\big)\right]\\
        &\leq \mathbb{E}\left[U_P^M\big(\ell\left(x^{\lambda^{b,M}}(T)\right)-U_A^{-1}(y^{\lambda^{b,M}}(T))\big)\right] - \mathbb{E}\left[U_P\big(\ell\left(x^{\lambda^{b,M}}(T)\right)-U_A^{-1}(y^{\lambda^{b,M}}(T))\big)\right] + \hat{\epsilon}\\
        &= \mathbb{E}\left[\big(U_P^M(z(T)) - U_P(z(T))\big)\mathbf{1}_{\{z(T)<-M\}}\right] + \hat{\epsilon}\\
        &\leq C_P\mathbb{E}\left[e^{-c_P z(T)}\mathbf{1}_{\{z(T)<-M\}}\right] + \hat{\epsilon}\\
        &\leq C_P R(M) + \hat{\epsilon}.
    \end{align}
                Since $\lambda^{b,M}$ is bounded by $\tilde{c}_\lambda$ independently of $M$ and $\epsilon$, and $0<\epsilon<1$, the drift and diffusion coefficients of $z$ are uniformly bounded in both $M$ and $\epsilon$. Hence, the constants $C_P$ and $\delta$ (from Dai et al. \cite{dai2026lifetime}), and therefore $R(M)$, can be chosen independently of $\epsilon$. Letting $\hat{\epsilon}\downarrow0$ in the preceding estimate yields
    \[
    0\leq
    \sup_{0<\epsilon<1}
    \left(w^{\epsilon,M}(x,y,t)-w^\epsilon(x,y,t)\right)
    \leq C_P R(M)\xrightarrow[M\to\infty]{}0.
    \]
    Therefore, the convergence is uniform in $\epsilon$.
\end{proof}
Combining Lemmas \ref{lemma_cvg_w^M0} and \ref{lemma_cvg_wM_to_w}, the Moore–Osgood theorem implies that both iterated limits exist and
\[
\lim_{\epsilon\to0}w^\epsilon
=
\lim_{\epsilon\to0}\lim_{M\to\infty}w^{\epsilon,M}
=
\lim_{M\to\infty}\lim_{\epsilon\to0}w^{\epsilon,M}
=
\lim_{M\to\infty}w^{0,M}.
\]
Furthermore, since $U_P^M\geq U_P$ by construction, $w^{0,M}\geq w^0$ for every $M$. Therefore
\begin{equation} \label{w^0,M_geq_w^0}
    \lim_{M\to\infty} w^{0,M} \geq w^0.
\end{equation}
Combining this inequality with \eqref{w^epsilon_leq_w^0}, we conclude that
$$
\lim_{\epsilon\to0}w^\epsilon = w^0.
$$
\begin{figure}[H]
\centering
\begin{tikzpicture}[scale=0.7, transform shape,
    block/.style={rectangle, draw=black, thick, fill=white,
                  text width=3.2cm, minimum height=1cm, align=center,
                  inner sep=8pt},
    arrow/.style={thick, -Stealth[scale=1.2]},
    doublearrow/.style={thick, <->, -Stealth[scale=1.2]},
    label/.style={midway, fill=white, inner sep=2pt, font=\small}
]
\node[block] (A) at (0,0) {$\lim_{\epsilon \rightarrow 0} w^\epsilon$};
\node[block, right=4.4cm of A] (B) {$\lim_{\epsilon \rightarrow 0}\lim_{M \rightarrow \infty} w^{\epsilon,M}$};
\node[block, right=4cm of B] (C) {$\lim_{M \rightarrow \infty}\lim_{\epsilon \rightarrow 0} w^{\epsilon,M}$};
\node[block, below=1.5cm of C] (D) {$\lim_{M \rightarrow \infty} w^{0,M}$};
\node[block, left=4cm of D] (E) {$w^0$};
\draw[thick, <->] (A) -- node[label, above] {Lemma \ref{lemma_w^bM_to_w^b} \& Theorem \ref{thm_cvg_w_i_w}} (B);
\draw[thick, <->] (B) -- node[label, above] {Lemma \ref{lemma_cvg_wM_to_w}} (C);
\draw[thick, <->] (C) -- node[label, left] {Lemma \ref{lemma_cvg_w^M0}} (D);
\draw[thick, ->] (D) -- node[label, above] {$\leq$ (see \eqref{w^0,M_geq_w^0})} (E);
\end{tikzpicture}
\caption{Logic flow}
\label{logic_2}
\end{figure}
Moreover, from the dominated convergence theorem, $w^0 = v \in B_C$. Thus, Step 1 is completed.

\textbf{Step 2:} Since $w^{\epsilon_m}\rightharpoonup w^0$ in $W^{2,1}_p(\overline{Q})$ for some sequence $\epsilon_m\downarrow0$, Mazur's lemma yields finite convex combinations
\[
    n_m:=\sum_{i=m}^{N_m}a_i^{(m)}w^{\epsilon_i}, \qquad a_i^{(m)}\geq0,\qquad \sum_{i=m}^{N_m}a_i^{(m)}=1,
\]
such that $n_m\to w^0$ strongly in $W^{2,1}_p(\overline{Q})$. Notice that $\|n_m\|_{W^{2,1}_\infty(\overline{Q})} \leq C_Q$ follows immediately from Theorem \ref{thm_cvg_w_i_w}(iii).

It follows that $\|n_m - w^0\|_{L^p(\overline{Q})} \rightarrow 0$. Then there exists a subsequence of $n_m$ such that $n_{m_j}(x,y,t) \rightarrow w^0(x,y,t)$ as $j \rightarrow \infty$ almost everywhere in $\overline{Q}$. Since $|n_{m_j}(x,y,t)| \leq C_Q$ everywhere in $\overline{Q}$ for every $j$, it follows that $|w^0(x,y,t)| \leq C_Q$ almost everywhere in $\overline{Q}$. Therefore, $\|w^0\|_{L^\infty(\overline{Q})} \leq C_Q$. Since strong convergence in $W^{2,1}_p(\overline{Q})$ guarantees strong convergence in $L^p(\overline{Q})$ for all partial derivatives, the same argument applies to $w^0_t$, $w^0_x$, $w^0_y$, $w^0_{xx}$, $w^0_{xy}$, and $w^0_{yy}$, allowing us to conclude that $w^0 \in W^{2,1}_{\infty}(\overline{Q})$.
\end{proof}

\end{document}